\documentclass[final,12pt]{elsarticle}
\usepackage{srcltx}
\usepackage{eurosym}
\usepackage{mathtools}
\usepackage{amsmath}
\usepackage{cases}
\usepackage{amsfonts}
\usepackage{amssymb}
\usepackage{amsthm}
\usepackage[utf8]{inputenc}
\usepackage{chngcntr}
\usepackage{apptools}

\usepackage[dvipsnames]{xcolor}

\usepackage{graphicx}
\usepackage{mathrsfs}
\usepackage{xcolor}
\usepackage{xpatch}
\usepackage{exscale}
\usepackage{latexsym}
\usepackage{tikz}
\usepackage{caption}
\usepackage{environ}

\usepackage{cancel,ulem}
\xpatchcmd{\MaketitleBox}{\hrule}{}{}{}% remove first horizontal rule (above abstract)
\xpatchcmd{\MaketitleBox}{\hrule}{}{}{}
\usepackage[colorlinks,plainpages=true,pdfpagelabels,hypertexnames=true,colorlinks=true,pdfstartview=FitV,linkcolor=blue,citecolor=red,urlcolor=black]{hyperref}
\PassOptionsToPackage{unicode}{hyperref}
\PassOptionsToPackage{naturalnames}{hyperref}
\usepackage{enumerate}
\usepackage[shortlabels]{enumitem}
\usepackage{bookmark}
\usepackage{wasysym}
\usepackage{esint}
\usepackage[ddmmyyyy]{datetime}
\usepackage[margin=2cm]{geometry}
\makeatletter
\g@addto@macro\normalsize{%
	\setlength\abovedisplayskip{4pt}
	\setlength\belowdisplayskip{4pt}
	\setlength\abovedisplayshortskip{4pt}
	\setlength\belowdisplayshortskip{4pt}
}
\numberwithin{equation}{section}
\everymath{\displaystyle}
\usepackage[capitalize,nameinlink]{cleveref}
\crefname{section}{Section}{Sections}
\crefname{subsection}{Subsection}{Subsections}
\crefname{condition}{Condition}{Conditions}
\crefname{hypothesis}{Hypothesis}{Conditions}
\crefname{assumption}{Assumption}{Assumptions}
\crefname{lemmaa}{Lemma}{Lemmas}
\crefname{lemma}{Lemma}{Lemmas}
\crefname{definition}{Definition}{Definitions}
\crefname{figure}{figure}{figures}
\crefname{question}{Question}{Questions}

\crefformat{equation}{\textup{#2(#1)#3}}
\crefrangeformat{equation}{\textup{#3(#1)#4--#5(#2)#6}}
\crefmultiformat{equation}{\textup{#2(#1)#3}}{ and \textup{#2(#1)#3}}
{, \textup{#2(#1)#3}}{, and \textup{#2(#1)#3}}
\crefrangemultiformat{equation}{\textup{#3(#1)#4--#5(#2)#6}}%
{ and \textup{#3(#1)#4--#5(#2)#6}}{, \textup{#3(#1)#4--#5(#2)#6}}%
{, and \textup{#3(#1)#4--#5(#2)#6}}

\Crefformat{equation}{#2Equation~\textup{(#1)}#3}
\Crefrangeformat{equation}{Equations~\textup{#3(#1)#4--#5(#2)#6}}
\Crefmultiformat{equation}{Equations~\textup{#2(#1)#3}}{ and \textup{#2(#1)#3}}
{, \textup{#2(#1)#3}}{, and \textup{#2(#1)#3}}
\Crefrangemultiformat{equation}{Equations~\textup{#3(#1)#4--#5(#2)#6}}%
{ and \textup{#3(#1)#4--#5(#2)#6}}{, \textup{#3(#1)#4--#5(#2)#6}}%
{, and \textup{#3(#1)#4--#5(#2)#6}}

\crefdefaultlabelformat{#2\textup{#1}#3}
\numberwithin{equation}{section}

\newtheorem{theorem} {Theorem}[section]
\newtheorem{proposition} {Proposition}[section]

\newtheorem{lemma}{Lemma}[section]

\newtheorem{counter-example}{counter-example}[section]
\newtheorem{remark} {Remark}[section]
\newtheorem{definition} {Definition}[section]
\newtheorem{question} {Question}[section]
\newtheorem{innercustomgeneric}{\customgenericname}
\providecommand{\customgenericname}{}
\newcommand{\newcustomtheorem}[2]{%
	\newenvironment{#1}[1]
	{%
		\renewcommand\customgenericname{#2}%
		\renewcommand\theinnercustomgeneric{##1}%
		\innercustomgeneric
	}
	{\endinnercustomgeneric}
}

\newcustomtheorem{customthm}{Theorem}
\newcustomtheorem{customlemma}{Lemma}

\def\CC{{\rm \kern.24em \vrule width.02em height1.4ex depth-.05ex \kern-.26emC}}

\def\TagOnRight

\def\AA{{it I} \hskip-3pt{\tt A}}

\def\QQ{\rlap {\raise 0.4ex \hbox{$\scriptscriptstyle |$}} {\hskip -0.1em Q}}

\makeatletter
\newcommand{\vo}{\vec{o}\@ifnextchar{^}{\,}{}}
\makeatother
\def\YYint#1#2#3{{\setbox0=\hbox{$#1{#2#3}{\iint}$}
		\vcenter{\hbox{$#2#3$}}\kern-.50\wd0}}
\def\XXint#1#2#3{{\setbox0=\hbox{$#1{#2#3}{\int}$}
		\vcenter{\hbox{$#2#3$}}\kern-.50\wd0}}

\makeatletter
\def\namedlabel#1#2{\begingroup
	\def\@currentlabel{#2}%
	\label{#1}\endgroup
}
\makeatother
\makeatletter
\newcommand{\rmh}[1]{\mathpalette{\raisem@th{#1}}}
\newcommand{\raisem@th}[3]{\hspace*{-1pt}\raisebox{#1}{$#2#3$}}
\makeatother

\newcommand{\descitem}[2]{\item[#1] \label{#2}}
\newcommand{\descref}[2]{\hyperref[#1]{\textnormal{\textcolor{black}{}\textcolor{blue}{\bf #2}\textcolor{black}{}}}}
\newcommand{\dref}[2]{\hyperref[#1]{\textcolor{black}{(}\textcolor{blue}{\bf #2}\textcolor{black}{)}}}
\newcommand{\be} {\begin{eqnarray}}
	\newcommand{\ee} {\end{eqnarray}}
\newcommand{\Bea} {\begin{eqnarray*}}
	\newcommand{\Eea} {\end{eqnarray*}}
\newcommand{\rr}{\rightarrow}

\newcommand{\B} {\beta}

\newcommand{\g} {\gamma}

\newcommand{\p}  {\prime}
\newcommand{\e}  {\epsilon}

\newcommand{\la} {\lambda}

\newcommand{\f}{\infty}

\newcommand{\R}{\mathbb{R}}

\newcommand{\al}{\alpha}
\newcommand{\ga}{\gamma}
\DeclareMathOperator{\sign}{sign}

\newcommand{\norm}[1]{\left|\hspace{-0.2mm}\left| #1 \right|\hspace{-0.2mm}\right|}
\newcommand{\abs}[1]{\left| #1\right|}

\newcommand{\RN}[1]{%
	\textup{\uppercase\expandafter{\romannumeral#1}}%
}

\newcounter{whitney}
\refstepcounter{whitney}

\newcounter{ineqcounter}
\refstepcounter{ineqcounter}
\makeatletter
\def\ps@pprintTitle{%
	\let\@oddhead\@empty
	\let\@evenhead\@empty
	\def\@oddfoot{}%
	\let\@evenfoot\@oddfoot}
\makeatother
\usepackage[doublespacing]{setspace}
\usepackage[titletoc,toc,page]{appendix}

\usepackage{pgf,tikz}
\usetikzlibrary{arrows}
\usetikzlibrary{decorations.pathreplacing}
\allowdisplaybreaks
\usepackage{cancel,soul,ulem}

\usepackage{etoolbox}
\usepackage{hyperref}

\usepackage{lipsum}% just for the example

\makeatletter
\def\@mkboth#1#2{}
\newlength\appendixwidth
\preto\appendix{\addtocontents{toc}{\protect\patchl@section}}
\newcommand{\patchl@section}{%
	\settowidth{\appendixwidth}{\textbf{Appendix }}%
	\addtolength{\appendixwidth}{1.5em}%
	\patchcmd{\l@section}{1.5em}{\appendixwidth}{}{\ddt}%
}
\makeatother

\begin{document}
	
	\begin{frontmatter}
		
		\title{{Fractional regularity for conservation laws with discontinuous flux}
		}

		% \author[myaddress]{Adimurthi\tnoteref{thanksfirstauthor}}
		% \cortext[mycorrespondingauthor]{Corresponding author}
		% \ead{aditi@math.tifrbng.res.in and adiadimurthi@gmail.com}
		% \tnotetext[thanksfirstauthor]{Supported in part by Rajaramanna Fellowship.}
		\author[myaddress]{Shyam Sundar Ghoshal}
		\ead{ghoshal@tifrbng.res.in}

		\author[myaddress1]{St\'ephane Junca}
		\ead{stephane.junca@univ-cotedazur.fr}

		\author[myaddress]{Akash Parmar}
		\ead{akash@tifrbng.res.in}
		\address[myaddress]{Centre for Applicable Mathematics, Tata Institute of Fundamental Research, Post Bag No 6503, Sharadanagar, Bangalore - 560065, India.}
		
		\address[myaddress1]{Universit\'e C\^ote d'Azur,  LJAD,  Inria \& CNRS,
			Parc Valrose,
			06108 Nice, 
			France.}
		
		%\date{\today}
		%\begin{center}{January 09, 2022} \end{center} 
		
		\begin{abstract}
			This article deals with the regularity  of the  entropy solutions  of  scalar conservation laws with discontinuous flux. It is well-known [Adimurthi et al.,  Comm. Pure Appl. Math. 2011] that the entropy solution for such equation  does not admit  $BV$ regularity in general, even when the initial data belongs {to} $BV$. Due to this phenomenon fractional $BV^s$  spaces  wider  than  $BV$  are required, where the exponent  $0 < s \leq 1$   and   $BV=BV^1$.  %From $L^{\f}$ to $BV$ regularizing effect is not true for the discontinuous flux and 
			It is a long standing open question to find the optimal regularizing effect for the discontinuous flux with  $L^{\f}$  initial data. 
			%This paper resolves these issues completely.
%			We prove, under  non-degeneracy conditions on the flux, that 
%			the solution belongs to  a fractional $BV$ space.
			% if the initial data is in a fractional $BV$ space then the solution   belongs to  another fractional  $BV$ space. 
			The optimal regularizing effect in $BV^s$ %when the initial data is only in $L^\f$ 
			is  {proven  on an important case using control theory}. % shown with a   counter-example. 
			The fractional exponent $s $  is at most $1/2$ even when the fluxes are uniformly convex.

		\end{abstract}

	\end{frontmatter}
%	\begin{center}{January 09, 2022} \end{center} 
	%
	%
	\tableofcontents
	\section{Introduction}
	This article deals with the regularity aspects of the solution for the following scalar conservation law with discontinuous flux:  
	\begin{eqnarray}\label{1.1}
		\left\{\begin{array}{rlll}
			u_{t}+f(u)_{x}&=0, &\mbox{ if }& x>0,  t>0,\\
			u_{t}+g(u)_{x}&=0,&\mbox{ if }& x<0,  t>0,\\
			u(x,0)&=u_{0}(x),  &\mbox{ if }& x\in\R,
		\end{array}\right.
	\end{eqnarray}
	where $u:\R\times[0,\f)\rr\R$ is unknown, $u_0(\cdot)\in L^{\f}(\R)$ is the initial data and the fluxes $f$, $g$ are  $C^{1}(\R)$ and  strictly convex (that means $f^\prime,g^\prime$ are increasing functions).
	
	The conservation law \eqref{1.1} arises in several frameworks of  real-life phenomena, physical situations and applied subjects. For example, the equation \eqref{1.1} occurs naturally in the two-phase flow of a heterogeneous porous medium in the petroleum reservoir \cite{petro}.  The equation \eqref{1.1} is also useful to understand the ideal clarifier thickener \cite{burger2}, traffic flow model with varying road surface conditions \cite{traffic} and ion etching accustomed for semiconductor industry \cite{ion}. The above examples are just a little glance at the broad applicability of the equation \eqref{1.1} in the fields of applied sciences. For more details, one can see \cite{burger2, burger4, diehl, diehl4}.
	
	%In consideration of its application, it is important to study conservation law \eqref{1.1}.  
	The equation \eqref{1.1} does not have a global classical solution even for smooth initial data, so one needs to consider the following notion of a weak solution: 
	\begin{definition}[{\bf Weak solution}]\label{weak}  
		A function $u\in C(0,T;L^{1}_{loc}(\R))$ is said to be a weak solution of the problem \eqref{1.1}  if
		\begin{equation*}
			\int\limits_{0}^{\f}\int\limits_{\R} u\frac{\partial \phi}{\partial t}+F(x,u)\frac{\partial \phi}{\partial x}\mathrm{d}x \ \mathrm{d}t +\int\limits_{\R}^{} u_{0}(x)\phi(x,0)\mathrm{d}x =0,
		\end{equation*}
		for all $\phi\in C^\f_c(\R\times\R^+)$, where  the flux  $F(x,u)$ is given as $F(x, u)=H(x)f(u)+(1-H(x))g(u)$, and $H(x)$ is   {the} Heaviside function.
		%	\begin{eqnarray*}
		%	\chi_{A}(x)=\left\{\begin{array}{llll}
		%	1 &\mbox{ if }& x\in A,\\
		%	0 &\mbox{ if }& x\notin A.
		%	\end{array}\right.
		%	\end{eqnarray*}
		
	\end{definition}
    From the above defined weak formulation it can be derived that if interface traces $u^\pm(t)=\lim\limits_{x\rr0\pm}u(x,t)$ exist then at $x=0$, $u$ satisfies Rankine-Hugoniot condition, namely, for almost all $t$
    \begin{equation}\label{RH}
    	f(u^{+}(t))=g(u^{-}(t)).
    \end{equation} 
   For equation \eqref{1.1}, the left and right traces $u^-,u^+$ play important roles in well-posedness theory and also in determining the regularity of solutions. In \cite{Kyoto}, authors proved the existence of the interface traces via Hamilton-Jacobi type equation.  % and they h
    %where $u^{+}(t)=\lim\limits_{x\rr0+}u(x,t)$ and $u^{-}(t)=\lim\limits_{x\rr0-}u(x,t)$.
    %
    %
    %
    %
    %Implicitly, the existence of left and right traces is assumed in the Rankine-Hugoniot condition \eqref{RH}. Note that the authors \cite{Kyoto} proved the existence of the traces near the interface via Hamilton-Jacobi type equation.  
    %
%
%
%    
%
%
%   

	%Existence of solutions to \eqref{1.1} has been extensively studied via various methods, e.g. (i) Hamilton-Jacobi method \cite{Kyoto}, (ii) vanishing viscosity method \cite{BGG} and (iii) numerical approximations \cite{AJG,boris1,burger1,KT,tower}. For conservation laws with discontinuous flux having countably many spatial jumps, existence of solutions has been proved in \cite{SAT,P09}. Unlike the case of scalar conservation laws, the solution .

	%The existence of traces like a $BV$ function is an important usual feature for entropy solutions of conservation laws \cite{P05,P07}.
	%In this paper, the  fractional $BV$ regularity of  solutions  is proven under nonlinear assumptions of the fluxes $f$ and $g$.  Therefore,  these traces exist \cite{MO1}.
	
	Because of the non-uniqueness of  weak solutions, one needs some extra condition called the ``entropy condition" to get the unique solution even for the case $f=g$ in \eqref{1.1}. For $f=g$, Kru\v{z}kov \cite{Kruzkov}  gave a generalized entropy condition and proved the uniqueness. But due to the discontinuity of flux at the interface, the Kru\v{z}kov entropy %\sout{does not work}
	is not good enough to prove the uniqueness of \eqref{1.1}.  Hence another condition is needed near the interface, called the ``interface entropy condition".

	Throughout this article, we use the following   notion of the entropy solution. 

		\begin{definition}[{\bf Entropy solution}, \cite{Kyoto}]
		A weak solution $u\in L^\f(\R\times[0,T])$  of the problem \eqref{1.1} is said to be an entropy solution if the following holds. 
		\begin{enumerate}
			\item $u$ satisfies Kruzkov entropy conditions on each side of the interface $x=0$, that is, in $\R\setminus\{0\}$.
			\item The interface traces $u^\pm(t)=\lim\limits_{x\rr0\pm}u(x,t)$ exist for almost all $t>0$ and they satisfy the  following ``interface entropy condition" for almost all $t>0$, 
			\begin{equation} \label{iec}
				|\{t: f'(u^{+}(t))>0> g'(u^{-}(t))\}|=0.
			\end{equation}
		\end{enumerate}
	\end{definition}
	%Note that  the above interface entropy condition uses the existence of traces at the interface.
	%However u
	Uniqueness has been proved in  \cite{Kyoto}  when interface traces exist for a weak solution and it satisfies the entropy condition \eqref{iec}. In the same article, the authors obtained  the useful Lax-Oleinik type explicit formulas for  (\ref{1.1}). The notion of `A-B entropy solution' is introduced in \cite{ASG} and it coincides with \eqref{iec} when $A=\theta_g,B=\theta_f$. The number $\theta_f$ is defined by   $ f(\theta_f)=\min f $ when $f$ admits a minimum and $g(\theta_g)=\min g$. Lax-Oleinik type formula is also available \cite{Explicit}  for the `A-B-entropy solutions'. It has been observed \cite{AS} that for the case $A<\theta_g$ or $B>\theta_f$, `A-B-entropy solutions' belong to BV space for BV initial data and for $A=\theta_g ,B=\theta_f $ total variation of entropy solution can blow up at finite time $t_0>0$ even for some BV initial data (see section \ref{sec:results-discont} for more details). Therefore, we work with the choice $A=\theta_g ,B=\theta_f $. 
 In this article, we rely on the interface entropy condition \eqref{iec}, 
 and we use the analysis of characteristics developed as  in \cite{Kyoto}.

	The well{-}posed theory from the numerical and theoretical aspects has been extensively studied. We refer to \cite{ASG,boris2,BGG,KT,P09}  and the references therein. 
	%Motivated from this there are several existence results available \cite{AJG, Explicit, Kyoto, boris1, burger3, SAT, gimse, tower, tower1}. These referred articles used different techniques and methods since these methods are interesting and appropriate to applications hence, it is good to illustrate these in a few more words.
	The existence of a solution of \eqref{1.1}  has been proved by several numerical schemes  \cite{AJG, boris1, SAT, tower}. Due to the absence of total variation bound of solutions even for  BV data, the singular mapping technique becomes useful to show the convergence of numerical schemes (see \cite{AJG,tower}). % In \cite{AJG, tower} they proved the existence of a solution for \eqref{1.1}  from a numerical scheme by using the solution of the Riemann problem given in \cite{Kyoto}. %In \cite{boris1}, the authors showed the simple way to calculate the numerical Godunov flux. 
	Very recently, in \cite{SAT} authors generalize the Godunov type scheme in the case when discontinuities of flux may have limit point even  when the set of discontinuities is dense. %They consider the Audusse-Perthame the adapted entropy and showed the convergence to adapted entropy solution. 
	 %In \cite{Explicit}, the authors have obtained the \tcb{(almost)} explicit Hopf-Lax type formula for the discontinuous flux, whose derivative leads to the solution of \eqref{1.1}. 

Due to the lack of the BV regularity of the entropy  solution of \eqref{1.1}, one needs to study the regularity aspects of the solution in some bigger space than $BV$. More precisely, in this paper, we quantify the sharp regularity of entropy solution of \eqref{1.1} in  suitable  fractional spaces.

		\subsection*{Structure of the paper}
	
	This paper is organized as follows: in sections \ref{sec:1.1} and \ref{sec:results-discont}, we have discussed regularity results for scalar conservation laws and for \eqref{1.1} respectively. Then, it leads to section \ref{sec:1.3} where we precisely state the regularity problems corresponding to the equation \eqref{1.1}.   %In the section * we motivate the reader about  the importance of the results on the present paper dividing into 3 subsections ** respectively. 
	% ** we discuss the regularity aspects of the conservation laws when the fluxes $f=g$. In subsection ** we describe the known regularity results for  discontinuous flux and in the subsection ** we discussed the questions which is  
	In section \ref{main} we describe our main results with   some remarks. To make this article self-contained, in section \ref{def} some definitions and preliminary results have been recalled from \cite{Kyoto, junca1}.  The detailed proofs of  main results are written in section \ref{mainproof}. Proofs of these results utilize the Hopf-Lax type formula and some results from \cite{Kyoto} and techniques from \cite{AS, S}. Proofs here are a little more involved technically.  In the last section, the construction of a counter-example  shows that the main results of the present article cannot be improved.
Two appendixes contain basic useful lemmas and explanations regarding our adaptation of the result from control theory \cite{AG}.
	
	\subsection{Optimal regularity results in  $BV^s$ spaces  for a  smooth  flux: $f=g$ 
	}\label{sec:1.1}
	%---------------------------------------------------------------------------------------------

	In this subsection, we discuss the entropy solution of \eqref{1.1} for $f=g$. Even for Lipschitz flux, the theory is well-posed \cite{godunov, Kruzkov, lax1, ol} in the  $L^\infty$ setting and  many methodologies  are available  to understand  the  regularity of entropy solutions \cite{finer,junca1,CJLO,CJJ,SAJJ,SA,GJC,lax1,P05,P07,ol}. 
	%Here, we would like to quantify the regularity of entropy solutions in a functional space framework that well represents the structure of entropy solutions.

	Natural  function space  for scalar conservation law is  $BV$  since  the fundamental work  of   A. I. Volpert \cite{Volpert} in 1967. This space allows to get compactness and it makes convenient  to describe the structure of  a shock wave  with traces on each side of the singularity \cite{AFP}.  Information on trace helps to study finer qualitative properties of solutions. 
	%The traces property is fundamental, for example to write the Rankine-Hugoniot condition through the shock.
	The occurrence of  the $BV$  regularity for entropy solution  appeared for the first time  in \cite{lax1,ol} independently by P. D. Lax and O. Oleinik. The entropy solution becomes instantaneously $BV$ even when the data is in $L^\infty$ and the flux is uniformly convex, i.e., $\inf f'' > 0$.  This is the well known smoothing effect  as a consequence of the one-sided Lipschitz-Oleinik inequality  \cite{ol}.    
	
	Unfortunately,  the `$BV$ space is not enough' \cite{Cheng83} when the flux is not uniformly convex. There are many examples of  entropy solutions that are not in $BV$ \cite{finer,CJ1,SAJJ}. Though non-vanishing property of second derivative of the flux is proved \cite{SA} to be necessary and sufficient condition for BV regularizing,   but a smoothing effect  still occurs in fractional  Sobolev spaces \cite{JaX,LPT} for a  nonlinear flux. To keep the advantages of space $BV$: regularity and some traces, the fractional $BV$ spaces were introduced for conservation laws in \cite{junca1}.  The Lax-Oleinik smoothing effect was generalized  in $BV^s$  for  a  flux with power-law nonlinearity like $ |u|^{p+1}$ and $p=1/s \geq 1$, %\sout{  for  a  $C^3$ flux   in \cite{EM} and }
	for $C^1$ or  strictly convex flux in  \cite{junca1,CJLO,GJC}.  
	%In a convex case, the smoothing effect  is a consequence of a generalized  one-sided Lipschitz-Oleinik inequality   written for the speed of the wave \cite{CJJ,GJC}.  
	%Moreover, this smoothing effect is sharp \cite{CJ1,SAJJ}.  It is the reason why the $BV^s$ setting are used  for a discontinuous flux for  the first time.
	%	 In particular,  it allows  to understand the loss of the $BV$ smoothing effect of Lax and Oleinik in the presence of an interface when $f$ and $g$ are uniformly convex. 
	
	Fractional BV spaces, denoted by $BV^{s}$, $0<s\le 1$, were  first defined for all $s\in (0,1)$ in \cite{love, MO, MO1}.
	Let $I$ be a non-empty interval of $\R$ and $s\in (0,1]$. The space of fractional bounded variation functions, denoted as $BV^{s}(I)$  is a generalization of  space of functions with a bounded variation on $I$, denoted as $BV(I)$.
	In the sequel, we denote $S(I)$ the set of the subdivisions of $I$, that is the set of finite subsets
	$\sigma =({x_{0},x_{1},...,x_{n}})$ in $I$ with $(x_{0}{<} x_{1} {<} x_{2} {<} ...  {<} x_{n})$.
	\begin{definition}[$BV^{s}$ \cite{love, MO, MO1}]
		Let  $\sigma=(x_{0},x_{1},...,x_{n})$ be in $S(I)$ and let $u$ be real function on $I$. The s-total variation of $u$ with respect to $\sigma$ is
		
		\begin{equation*}
			TV^{s}u({\sigma})=\sum_{i=1}^{n}{|u(x_{i})-u(x_{i-1})|}^{1/s},
		\end{equation*}
		then define,
		\begin{equation*}
			TV^{s}u(I)= \sup_{\sigma\in S(I)} TV^{s}u({\sigma}).
		\end{equation*}
		The set $BV^{s}(I)$ is the set of functions $u:I\rr\R $ such that $TV^{s}u(I)<\f$.
	\end{definition}
	
	%For $f=g$,  \eqref{1.1} the theory is well-posed \cite{godunov, Kruzkov, lax1, ol} and also regularity aspects are  well understood \cite{AS, finer,junca1, junca2, SAJJ, SA, lax1, ol} and put other papers.% There are several methods to show the existence of the solution to \eqref{1.1} such as vanishing viscosity \cite{Kruzkov}, Hamilton-Jacobi method \cite{lax1} and numerical scheme methods \cite{godunov}. On the other hand there are several regularity results, like in \cite{lax1} it has been proved that for given initial data in $L^{\f}$ produces a BV$_{loc}$ solution. For finer regularity results one can see \cite{finer} here they uses the structure theorem from \cite{structure} to prove that although, $u$ may not be in BV$_{loc}(\R)$, but the composition of $u$ with the derivative of flux is special bounded variation function, denoted by SBV$(\R)$.

	%On a slightly different note regularity results are not true in multi-dimensions. In \cite{SA} authors gave the explicit construction of entropy solution to prove the non existence of regular entropy solution. Furthermore they also showed that even uniformly convexity and the degeneracy condition are not enough to get regularity in multi-dimension.

	%---------------------------------------------------------------------
	\subsection{Previous regularity results for   discontinuous flux}\label{sec:results-discont} % the solution of (\ref{1.1})}
	To study the convergence and existence of traces of the solution, BV regularity plays a very important role. Without total variation bound, the convergence of numerical scheme is difficult. It is to be noted that one can not expect  the total variation diminishing of the solution to \eqref{1.1} since non-constant solution can arise from constant initial data. %because some constant initial data does give a non-constant solution which implies that total variation is increasing instead of decreasing.
	 Despite extensive study of the equation \eqref{1.1} for several decades, optimal regularity results were missing for the  solution of \eqref{1.1}. There are few results known about the regularity aspects of the solution to  \eqref{1.1} and we describe them in the following paragraph.

	Though away from interface solution has been proved \cite{burger1} to be BV in space, the regularity of solution near interface was unknown for long time. First breakthrough result \cite{AS} comes in 2009. By constructing explicit example when $\min f\neq \min g$, authors \cite{AS} have shown that total variation of entropy solution to \eqref{1.1} blows up at time $t_0>0$ for BV initial data. To build the example, they have utilized the failure of Lipschitz continuity of $f^{-1}g$ near the critical point of $f$. Here $g_{-}^{-1}$, $f_{+}^{-1}$ are the inverse of $g, f$ in appropriate domains, more precisely, they are defined as %   In \cite{burger1}, the authors proved that away from the interface the total variation for solution of \eqref{1.1} is bounded, i.e., solution is a BV function away from the interface.
	%A few years later a breakthrough result given in \cite{AS}, contrary to expectations, authors showed that there exists an initial data in BV space such that the corresponding entropy solution does not belong to BV, although the fluxes $f$ and $g$ are uniformly convex. To understand more about interface traces, let use denote  $g_{-}^{-1}$, $f_{+}^{-1}$ by the inverse of $g, f$ for  the domains where $g'(u)\le0$ and $f'(u)\ge 0$ respectively,
	\begin{equation}\label{inverses}
		g_{-}^{-1}: \; \; ]  (g')^{-1} ( - \infty),   (g')^{-1} (0)] \rightarrow \R \qquad  \&  \qquad   f_{+}^{-1}: [(f')^{-1} (0), (f')^{-1} (+ \infty)[   \rightarrow \R.
	\end{equation}
   The key  functions $f^{-1}_{+}g(\cdot)$ and $g^{-1}_{-}f(\cdot)$ transmit  information via the interface from left-to-right and right-to-left respectively.
%To introduce the singular mapping technique they denote  $g_{-}^{-1}$, $f_{+}^{-1}$ by the inverse of $g, f$ for  the domains where $g'(u)\le0$ and $f'(u)\ge 0$ respectively,
%	\begin{equation}\label{inverses}
%		g_{-}^{-1}: \; \; ]  (g')^{-1} ( - \infty),   (g')^{-1} (0)] \rightarrow \R \qquad  \&  \qquad   f_{+}^{-1}: [(f')^{-1} (0), (f')^{-1} (+ \infty)[   \rightarrow \R.
%	\end{equation}
%	The singular mapping technique uses the key  functions $f^{-1}_{+}g(\cdot)$ and $g^{-1}_{-}f(\cdot)$ which transmit  information via the interface. \tcr{ They also showed that if $f^{-1}_{+}g(\cdot)$ and $g^{-1}_{-}f(\cdot)$ are not singular, indeed Lipschitz continuous functions, then solution of \eqref{1.1} admits the BV regularity. }
	
	On the other hand in \cite{S,S2}  the author proved several regularity results. 
	More surprisingly, the author showed that one can prove that solution of \eqref{1.1} belongs to $BV$ if fluxes have the same minimum value, i.e., $ \min f=f(\theta_{f})= \min g  = g(\theta_{g})$. The author also proved that if  $f(\theta_{f})\not=g(\theta_{g})$ and initial data is compactly supported then there exists a time $T$ such that for all $t>T$ solution of \eqref{1.1} admits the BV regularity. The assumption of compact support can not be relaxed as it has been shown by example that there exists a sequence of time, $T_{n}$, for which the total variation of solution to \eqref{1.1} blows up. %\textcolor{teal}{For  $f(\theta_{f})\not=g(\theta_{g})$ the author also gave a counter-example which is much more generalized than the one is given in \cite{AS}. Author, in \cite{AS}, proves that BV blow up at some fix time whereas the results in \cite{S} proves that there exists a sequence of time, $T_{n}$, for which the total variation of solution to \eqref{1.1} blows up which also showed that these results in \cite{S} are optimal.}
	% \begin{theorem}
	% Let $f$ and $g$ be the uniformly convex fluxes with $f(\theta_{f})\not=g(\theta_{g})$. Let $u_{0}\in L^{\f}$ be the initial data corresponding to solution $u$ of \eqref{1.1} and $u_{0}$ is compactly supported. Then there exists a $T>0$ such that for all $t>T$
	%\begin{equation}
	%TV(u(\cdot, t), I(M))\le C(M, t),
	%\end{equation}
	%where $I(M)=\{x\in\R : |x|\le M\}$ and $C(M, t)>0$. Furthermore, we have
	%\begin{equation}
	%TV(u(\cdot, t), \R)\le C(t),
	%\end{equation}
	%where $C(t)>0$.
	%\end{theorem}
	%\begin{theorem}
	%Let $f$ and $g$ be the uniformly convex fluxes with $f(\theta_{f})=g(\theta_{g})$. Let $u_{0}\in L^{\f}$ be the initial data corresponding to solution $u$ of \eqref{1.1}. Then  for all $t>0$
	%\begin{equation}
	%TV(u(\cdot, t), I(R_{1}(t))\cup I(L_1(t)))\le C(M, t),
	%\end{equation}
	%where $I(R_{1}(t))=\{x\in\R : |x|\le R_1(t)\}$, $I(L_{1}(t))=\{x\in\R : |x|\le L_1(t)\}$ and $C(M, t)>0$. Furthermore, if $u_0$ is compactly supported then we have,
	%\begin{equation}
	%TV(u(\cdot, t), \R)\le C(t),
	%\end{equation}
	%where $C(t)>0$.
	%\end{theorem}
	%\begin{theorem}
	%Let $f$ and $g$ be the uniformly convex fluxes with $f(\theta_{f})=g(\theta_{g})$. Let $u_{0}\in BV(\R)$ be the initial data corresponding to solution $u$ of \eqref{1.1}. Then for all $t>0$
	%\begin{equation}
	%TV(u(\cdot, t))\le C(t)(TV(u_{0}+1)+4||u_{0}||_{\f},
	%\end{equation}
	%where $C(t)>0$.
	% \end{theorem}

	Earlier referred publications have uniform convexity assumption on the fluxes, in \cite{S2} it has been proved that even for non-uniform convex flux (but with a special structure when the flux losses its uniform convexity) any $L^{\f}$ initial data gives the solution which is $BV_{loc}$ near the interface when the connection $(A,B)$ as in \cite{Explicit}  are far from the critical point.
	\\
	
	\par This discussion leads to conclude that working in the BV space setup is not enough for scalar conservation law with discontinuous flux \eqref{1.1}. Hence, working in larger space appears appropriate and we work in the space of functions of fractional bounded variation, denoted as $BV^{s}$, which is  more generalised space than the BV space.
	
	In the following subsection we discuss the questions which are answered in the present paper.  

	\subsection{Questions on the $BV^s$ regularity for  discontinuous flux}\label{sec:1.3}% Discussions on the main results for (\ref{1.1})}
	As we discussed so far, {\it the entropy solution of ({\ref{1.1}}) lacks} the following properties: 
	\begin{itemize}
		\item [1.] If $u_0\in BV(\R)$ then $u(\cdot,t)\in BV(\R)$, for any $t>0$.
		\item [2.] If $f,g$ are uniformly convex fluxes,  { $\min f \neq \min g$} and $u_0\in L^\f(\R)$ then for any $t>0$, $u(\cdot,t)\in BV_{\mbox{loc}}$.
	\end{itemize}
	%By motivating from the above features
	By motivating the subject on the basis of the above facts,  we settle the following issues regarding regularity of the solution of \eqref{1.1}.
	\begin{question}\label{q1}
		Can we expect that   for  a  well chosen $ 0 < s \leq  1$,  if the given initial data belongs to  $BV^{s}$  then  the solution of \eqref{1.1} stays in   $BV^{s}$?
	\end{question}
	\begin{question}\label{q1bis}
		Can we expect that   for  any $ 0 < s  \leq 1$  there exists $0 < s_1$ such that  if the given initial data belongs to  $BV^{s}$ then the solution of \eqref{1.1} belongs  to $BV^{s_1}$? 
	\end{question}
	In particular,  when $u_0$ is in $BV=BV^1$, $s=1$, in which $BV^{s_1}$, $0< s_1 \leq 1$, is the solution?
	\begin{question}\label{q-reg}
		What is the Lax-Oleinik type regularizing effect, for uniformly convex fluxes $f$ and $g$? In other words, does entropy solution of \eqref{1.1} belong to $BV^{s}$ for some $s\in(0,1)$ and for any given $L^\f$ initial data? %  i.e. for given initial data $u_{0}\in L^{\f}$ does solution of \eqref{1.1} belongs to $BV^{s}$ for some $s$?  
	\end{question}
	\begin{question}\label{q3}
		Can we choose $0 <s < 1$  sharply  and an initial data $u_{0} \in BV^s$  space for which the  generalized total variation blows up for the corresponding solution of \eqref{1.1}?
	\end{question}
	We are able to  answer all questions \ref{q1}-\ref{q3} affirmatively under certain assumptions on the fluxes $f,g$. We then show by counter-examples that the assumptions of our main results are optimal. Moreover, we provide explicit estimates of $s$-total variation of the solution with respect to time variable $t$ with  some sufficient conditions on initial data.
	% Even for $s=1$ the explicit estimate is new in the context of scalar conservation laws with discontinuous flux.}
	
	% If we assume the non-degeneracy flux condition then for initial data in $L^{\f}$ we have shown that solution of \eqref{1.1} belongs to $BV^{s}$ for some $s$.
	%%\begin{equation}
	%%\frac{|f'(u)-f'(v)|}{|u-v|^{p}}>C_{1},\ \ \ \ \ \ \ \ \frac{|g'(u)-g'(v)|}{|u-v|^{q}}>C_{2},
	%%\end{equation} 
	%Here we are dealing with functions of fractional bounded variation, $BV^{s}(\R)$, instead of functions of bounded variation $BV(\R)$. After seeing the enough literature about the scalar conservation law \eqref{1.1} in the BV space leads us to work in larger space. The reason after this is $BV^{s}(\R)$ spaces allow to work with less regular functions than $BV(\R)$ functions and appear to be more natural in this context. The $BV^{s}(\R)$ also embedded in the fractional Sobolev space, so working with these spaces also gives us regularity of solution for \eqref{1.1} in the fractional Sobolev space. Here we gave a counter-example to show our results are optimal.
	% We know that for these spaces our conventional triangle inequality is not true, so         

	\section{Main Results}\label{main}
	Throughout the paper, $f$ and $g$ are $C^1$  strictly convex functions admitting a critical point.
	Let $\theta_{f}$ and $\theta_{g}$ be the unique critical points of $f$ and $g$ respectively, i.e., $f'(\theta_{f})=0$ and $g'(\theta_{g})=0$ and $g_{-}^{-1}$, $f_{+}^{-1}$ denote the inverse of $g, f$ for domain where $g'(u)\le0$ and $f'(u)\ge 0$ respectively.  Notice that the existence of a minimum for $f$ and $g$ are always assumed in this paper as it allows the critical behaviour of admissible solution. 
	If $f$ and $g$  have no minimum but both are strictly increasing or decreasing, the situation is simpler \cite{AS}.  Thus, throughout the paper, it is assumed that, 
	\begin{equation}\label{hyp:diff-min}
		f(\theta_f)=\min f \neq \min g = g(\theta_g).
	\end{equation}
	 In the best case when $f$ and $g$  are uniformly convex and \eqref{hyp:diff-min} is satisfied, we obtain a smoothing in $BV^{1/2}$ instead of $BV$.
	
	In the non uniformly convex case the situation is worse.  The smoothing depends on the nonlinear flatness of the fluxes. 
	Let us introduce the following {\it non-degeneracy flux condition}  which is a power-law degeneracy condition \cite{junca1}, 
	there exist  two numbers $ p\geq 1$, $q \geq 1$, such that,  for  any   compact set  $K$, there exist  positive numbers $C_1, C_2$ such that for all $u\not=v$, $u,v\in K$, 
	\begin{eqnarray}\label{fluxc}
		\frac{|f'(u)-f'(v)|}{|u-v|^{p}}>C_{1} >0  &\mbox{ and }&\frac{|g'(u)-g'(v)|}{|u-v|^{q}}>C_{2}>0 .
	\end{eqnarray} 
		For $p=1$ this is the classical uniformly convex condition  for $f$ and for $p>1$ it corresponds to a less nonlinear convex flux like $f(u)=|u|^{p+1}$.
	
		An interesting subcase is when the loss of uniform convexity  of the fluxes occurs only at the minimum.  That is, if $f$ belongs to $C^2$ that convex power laws, $f(u)=|u|^{p+1}$, $p>1$ are the typical example. The same assumption can be made for the other flux $g$. 
	\begin{equation}\label{f''}
		f^{\p\p},g^{\p\p}  \mbox{ vanish only at $\theta_f$ and $\theta_g$ respectively.}
	\end{equation}
	The assumption \eqref{f''} combined with the previous one \eqref{fluxc} is also called the {\it restricted non-degeneracy condition} {and the fluxes, restricted fluxes}.
	In the subcase \eqref{f''} both satisfied by $f$ and $g$, stronger results are obtained and first presented in following Theorem \ref{fnoteqg} for an $L^\infty$  initial data and   Theorem \ref{th:f-not-g} for a $BV^s$ initial data.
	Two quantities are fundamental to express the fractional regularity of the solutions, $\gamma$ and $\nu$,
		\begin{align} \label{eq:gamma}
				\gamma=\left\{\begin{array}{l}
					\frac{1}{q+1}  \\
					\frac{1}{p+1} 
				\end{array}\right.
			& 	&
					\nu=\left\{\begin{array}{l}
					\frac{1}{p} \\
					\frac{1}{q} 
				\end{array}\right.
			&	&
					\begin{array}{lll}
					\mbox{ if }& \min f < \min g, \\
					\\
					\mbox{ if }& \min f > \min g.
				\end{array}
			\end{align}
			The constant $\gamma \leq 1/2$ can be understood as a loss of regularity due to the interface and $\nu\le1$ as the smoothing effect outside the interface. More precisely, $\gamma$ comes from the singular mapping technique as explained in  following remark. 
	\begin{remark}
		Let $f$ and $g$ be the fluxes satisfying the non-degeneracy condition \eqref{fluxc} and $f(\theta_{f})\not=g(\theta_{g})$. Then one of $f^{-1}_{+}g(\cdot)$ and $g^{-1}_{-}f(\cdot)$  is Lipschitz continuous and the other one is H\"older continuous with exponent $\gamma$ depending on $p,q$ from the non-degeneracy condition \eqref{fluxc}   and  given   by \eqref{eq:gamma}. The proof of the above fact is done in Lemma \ref{lipschitz}.
	\end{remark}
	
%	\tcb{The constant $\nu \le 1$ is related to the smoothing effect without interface for the flux with the highest minimum.}
	\begin{theorem}[{\bf Smoothing effect for restricted nonlinear fluxes and $L^\infty$ initial data}]\label{fnoteqg}
		Let $f$ and $g$ be two $C^2$ fluxes satisfying the  restricted non-degeneracy condition
		$f(\theta_{f})\not=g(\theta_{g})$  \eqref{hyp:diff-min}, \eqref{fluxc} and \eqref{f''}.
		Let $u(\cdot, t)$ be  the entropy solution of \eqref{1.1} corresponding to an initial data $u_{0}\in L^{\f}(\R)$. Then, $u(\cdot,t)\in BV^{s}(-M,M)$ for each $t>0,\,M>0$, 
		where $s$ is determined as follows
		\begin{equation}\label{def:s}
			s=\min (\gamma,\nu)
		\end{equation}
	and the following estimate holds with a positive constant $C_{f,g,\norm{u_0}_\f}$ depending only on the fluxes and the range of the initial data, 
			\begin{equation}
				TV^{s}(u(\cdot,t),[-M,M])\leq C_{f,g,\norm{u_0}_\f} +3(2||u_{0}||_{\f})^{1/s}+\frac{C_{f,g,\norm{u_0}_\f}M}{t}.
			\end{equation}
	\end{theorem}
	\begin{remark}[{\bf Uniform convex fluxes and $BV^{1/2}$}]
	   If the fluxes $f$ and $g$ are uniformly convex then the solution belongs to $BV^{1/2}$. So even for the uniformly convex case the solution goes into a fractional $BV$ space.
	\end{remark}
 Hence in the following theorem for $BV^s$ initial data, $0<s \leq 1$ the previous result states as follow.  Indeed, previous Theorem \ref{fnoteqg} can be seen as a limiting case of following Theorem \ref{th:f-not-g} with $s=0$ and  stating $BV^0=L^\infty$.
	\begin{theorem}[{\bf Smoothing effect for restricted nonlinear fluxes and $BV^s$ initial data}]\label{th:f-not-g}
			Let $f$ and $g$ be two $C^2$ fluxes such that $f(\theta_{f})\not=g(\theta_{g})$ and fluxes satisfy the restricted  non-degeneracy condition \eqref{fluxc} and \eqref{f''}. %Assume that $f^{\p\p},g^{\p\p}$ vanish only at $\theta_f$ and $\theta_g$ respectively. 
			% {\color{blue}Assume that $f^{\p\p},g^{\p\p}$ vanish only at $\theta_f$ and $\theta_g$ respectively}. 
			Let $u(\cdot, t)$ be  the entropy solution of \eqref{1.1} corresponding to an initial data $u_{0}\in BV^{s}(\R)$ for $s\in(0, 1)$. Then, $u(\cdot,t)\in BV^{s_1}(-M,M)$ for each $t>0,\,M>0$ with 
			\begin{equation}
				s_1:=\min\{\ga,\max\{\nu,s\}\}
			\end{equation} 
			the following estimate holds with a positive constant $C_{f,g,||u_{0}||_{\f}}$ depending only on fluxes and the range of the initial data and a constant $D>0$,
			\begin{equation}
				TV^{s_1}(u(\cdot,t),[-M,M])\leq C_{f,g,\norm{u_0}_\f}+\frac{C_{f,g,\norm{u_0}_\f}M}{t}+2\norm{2u_0}_{\f}^{\frac{1}{s_1}}+D\cdot TV^s(u_0).
			\end{equation}
			%Moreover, we have the following global regularity result: $u(\cdot,t)\in BV^{s_1}(\R)$ for $t>0$ where $s_1=\min\{\ga,s\}$.
		\end{theorem}
	
		We note that the assumption on vanishing points of $f^{\p\p},g^{\p\p}$ is restrictive. We can relax this assumption at the cost of smaller $s_1$. More precisely, we have the following result. %if we assume $f,g$ to be $C^1$ functions satisfying \eqref{fluxc}, then the regularity exponent $s_1$ will be $$s_1=\gamma\min\{ \max\{p^{-1}, s\},\max\{q^{-1}, s\} \}$$ where $\ga$ is defined in \eqref{eq:gamma}.
		%
		%}
		\begin{theorem}[{\bf Smoothing effect for $L^\infty$ initial data}]\label{theorem:gen:reg} 
			Let $f$ and $g$ be two $C^2$ fluxes such that $f(\theta_{f})\not=g(\theta_{g})$ satisfying the non-degeneracy condition \eqref{fluxc} with exponent $p,q$ respectively. Let $u(\cdot, t)$ be  the entropy solution of \eqref{1.1} corresponding to an initial data $u_{0}\in L^{\f}(\R)$. Then, for each $t>0, M>0$ and there exists positive constant $C_{f,g,||u_{0}||_{\f}}$ such that 
			\begin{equation*}
				TV^{s}(u(\cdot,t), [-M, M])\leq C_{f,g,\norm{u_0}_\f}+3(2||u_{0}||_{\f})^{1/s}+\frac{C_{f,g,\norm{u_0}_\f}M}{t}
			\end{equation*}
			where $s$ is determined as follows
			\begin{equation}\label{def:s-1}
				s= \g \, \nu.
			\end{equation}
		\end{theorem}
	\begin{theorem}
		With the same assumption as Theorem \ref{theorem:gen:reg},  if $u_{0}\in BV^{s_{0}}$. Then, $u(\cdot, t)\in BV^{s_{2}}$ and there exists positive constants $C_{f,g,||u_{0}||_{\f}}$ and  $D$ such that 
		\begin{eqnarray}
			TV^{s_{2}}(u(\cdot, t),[-M, M])\le C_{f,g,\norm{u_0}_\f}+\frac{C_{f,g,\norm{u_0}_\f}M}{t}+2\norm{2u_0}_{\f}^{\frac{1}{s_2}}+D\cdot TV^{s_{0}}(u_0),
		\end{eqnarray}
	where,
	\begin{eqnarray}
		s_{2}=\gamma\max(s_{0},\nu).
	\end{eqnarray}
	\end{theorem}
	In general, away from the interface, the expected fractional regularity  %\sout{ expected without  the interface}
	is $\min(1/p,1/q)$ \cite{junca1} which is always bigger than $s$ in \eqref{def:s}.  
	 In particular, near the interface, for $BV$ initial data, a $BV$ regularity for the entropy solution cannot be expected \cite{AS}. At most, a $BV^{1/2}$ regularity is possible. Getting   $BV$ regularity of entropy solution can be impossible  near the interface. The situation is better far from the interface.
	
	 Far from the interface, the constant $\gamma$ plays no role. The following  theorem gives estimates which are  sharp  for small time.
	\begin{theorem}[{\bf  Regularity outside the interface}]\label{eps}
			Let $f$ and $g$ be the fluxes with $f(\theta_{f})\not=g(\theta_{g})$.  Let $u(\cdot, t)$ be  the entropy solution of \eqref{1.1} corresponding to an initial data $u_{0}\in BV^{s}(\R)$ for $s\in(0, 1)$. Then the following holds.
			\begin{enumerate}
				\item If $f,g$ satisfies \eqref{fluxc} with exponent $p,q$ respectively, for any $t>0$, $\epsilon > 0$, then there exists a constant $C_{f,g, ||u_{0}||_{\f}}>0$
				\begin{equation}\label{eps-1}
					TV^{s_1}(u(\cdot,t),(-\f, \e]\cup[\e, \f))\leq \frac{C_{f,g,\norm{u_0}_\f}t}{\e}+ 2TV^{s_1}(u_{0})+2(2||u_{0}||_{\infty})^{1/s_1}
				\end{equation}
				for $s_1=\min\{p^{-1},q^{-1},s\}$. 
				\item If we assume \eqref{f''} that $f^{\p\p},g^{\p\p}$ vanish only at $\theta_f$ and $\theta_g$ respectively then we have
				\begin{align}
					TV^{s}(u(\cdot,t),(-\f, \e]\cup[\e, \f))&\leq \frac{C_{f,g,\norm{u_0}_\f}}{\min\{f^{\p\p}(v);[(f^{\p})^{-1}(\e/t),S_{f,g,\norm{u_0}_\f}]\}}\frac{t}{\e}+TV^{s}(u_0)\nonumber\\
					&+\frac{C_{f,g,\norm{u_0}_\f}}{\min\{g^{\p\p}(v);[-S_{f,g,\norm{u_0}_\f},(g^{\p})^{-1}(-\e/t)]\}}\frac{t}{\e}+2(2||u_{0}||_{\infty})^{1/s}\label{eps-2}
				\end{align}
				for any $t>0,\e>0$ where $S_{f,g,\norm{u_0}_\f}$ is defined as 
				$$S_{f,g,\norm{u_0}_\f}=\max \left\{ \|u_0\|_\infty ,   \sup_{|v| \leq \|u_0\|_\infty }|f^{-1}_+ ( g(v))|,  \sup_{|v| \leq \|u_0\|_\infty }|g^{-1}_- ( f(v))|  \right\}.$$
			\end{enumerate}
		\end{theorem}

	\begin{remark}
		All the regularity results in Theorems \ref{fnoteqg},  \ref{th:f-not-g},  \ref{theorem:gen:reg},   \ref{eps} are extendable to fractional Sobolev space $W^{s,p}$ with the same exponent $s$,  up to any $\varepsilon>0$, thanks to  the embedding  $BV^s\subset W^{s-\varepsilon,1/s} $   for all $\varepsilon \in (0,s)$ \cite{junca1}.
		%for all $0<t<s$ or one can do direct calculation.
	\end{remark}
%	\tcb{  Now the optimality is discussed. The construction of counter-example prove the optimality of  Theorem \ref{th:f-not-g}. That means that }  the exponent $s_1$ in Theorem \ref{th:f-not-g} can not be improved.
	
	Now we discuss about  the optimality result. The assumption $\min f\not=\min g$ forbids the favourable case $f=g$, that is without interface.
	 Here, the optimality of Theorem \ref{th:f-not-g} is proved in the best case with uniformly convex fluxes. For this purpose  examples are built with the optimal regularity and not more up. The same construction is valid with a power law on one side of the interface.   These examples highlight the sharpness of Theorem \ref{th:f-not-g}.
	\begin{theorem}[{\bf Blow-up  for critical  $BV^s$ semi-norms}]\label{example}
		Let $p\geq1$ and $\e>0$. Then there exists fluxes $f,g$ and an initial data $u_0\in BV(\R)$ such that 
		\begin{enumerate}
			\item the flux $f$ satisfies the non-degeneracy condition \eqref{fluxc} with exponent $p$,
			\item the function $g$ is uniformly convex,
			\item the corresponding entropy solution $u(\cdot,T)\notin BV^s_{loc}(\R)$ for some $T>0$ and  $s=\frac{1}{p+1}+\e$. %\tcr{ in space for  a fixed positive time???} for
		\end{enumerate}
	\end{theorem}
	The proof of Theorem \ref{example} is postponed in  Section \ref{sec:opt} and \ref{ap:opt}.

	\section{Preliminaries}\label{def}
	The fundamental paper used here is \cite{Kyoto} where Adimurthi and Gowda settled an important  foundation of the theory on scalar conservation laws with an interface and two convex fluxes. 
	In this paper the author proposed the natural entropy condition \eqref{iec} at the interface which means that no information comes only from the interface but crosses or go towards the interface.   Such entropy condition is in the spirit of Lax-entropy conditions for shock waves. To make this paper self-contained, we recall some definitions and results.

	The following theorem can be found in \cite{Kyoto} Lemma 4.9 at page 51. It is  a Lax-Oleinik or Lax-Hopf formula for the initial value problem \eqref{1.1}.
	\begin{theorem}[\cite{Kyoto}]\label{constant}
		Let $u_{0}\in L^{\f}(\R)$, then there exists the entropy solution $u(\cdot, t)$ of \eqref{1.1} corresponding to an initial data $u_{0}$. Furthermore, there exist Lipschitz curves $R_{1}(t)\geq  R_{2}(t)\ge0$ and $L_{1}(t)\leq  L_{2}(t)\le0$, monotone functions $z_{\pm}(x,t)$ \textcolor{black}{non-decreasing in $x$ and non-increasing in $t$} and $t_{\pm}(x,t)$ \textcolor{black}{non-increasing in $x$ and non-decreasing in $t$} such that the solution $u(x,t)$ can be given by the explicit formula for  almost all $t>0$,
		\begin{eqnarray*}
			u(x,t)=\left\{\begin{array}{lllll}
				(f')^{-1}\left(\frac{x-z_+(x,t)}{t}\right)& \mbox{ if }& x\ge R_{1}(t),\\
				(f')^{-1}\left(\frac{x}{t-t_{+}(x,t)}\right) &\mbox{ if }& 0\le x<R_{1}(t),\\
				(g')^{-1}\left(\frac{x-z_-(x,t)}{t}\right) &\mbox{ if }& x\le L_{1}(t),\\
				(g')^{-1}\left(\frac{x}{t-t_{-}(x,t)}\right) &\mbox{ if }& L_{1}(t)<x<0.
			\end{array}\right.
		\end{eqnarray*}
		Furthermore, if $f(\theta_{f})\ge g(\theta_{g})$ then $R_1(t)=R_2(t)$ and if $f(\theta_{f})\le g(\theta_{g})$ then $L_1(t)=L_2(t)$. We also  have only three cases and following formula to compute the solution: 
		
		\begin{description}
			\descitem{Case 1:}{thm1-case-1} $L_{1}(t)=0$ and $R_{1}(t)=0$,
			\begin{eqnarray*}
				u(x, t)=\left\{\begin{array}{lllll}
					u_{0}(z_{+}(x, t))&\mbox{ if }& x>0,\\
					u_{0}(z_{-}(x, t))& \mbox{ if }& x<0.
				\end{array}\right.
			\end{eqnarray*}
			\descitem{Case 2:}{thm1-case-2} $L_{1}(t)=0$ and $R_{1}(t)>0$, then
			\begin{eqnarray*}
				u(x,t)=\left\{\begin{array}{lllll}
					f_{+}^{-1}g(u_{0}(z_{+}(x,t)))&\mbox{ if }& 0<x<R_{2}(t),\\
					f_{+}^{-1}g(\theta_{g}) &\mbox{ if }& R_{2}(t)\le x\le R_{1}(t),\\
					u_{0}(z_{-}(x,t))& \mbox{ if }& x<0.
				\end{array}\right.
			\end{eqnarray*}
			\descitem{Case 3:}{thm1-case-3} $L_{1}(t)<0$, $R_{1}(t)=0$, then
			\begin{eqnarray*}
				u(x,t)=\left\{\begin{array}{lllll}
					g_{-}^{-1}f(u_{0}(z_{-}(x,t)))&\mbox{ if }& L_{2}(t)<x<0,\\
					u_{0}(z_{-}(x, t))&\mbox{ if }&  x\le L_{1}(t),\\
					g_{-}^{-1}f(\theta_{f})& \mbox{ if }& L_{1}(t)<x<L_{2}(t).
				\end{array}\right.
			\end{eqnarray*}
		\end{description}
	\end{theorem}
	\begin{center}
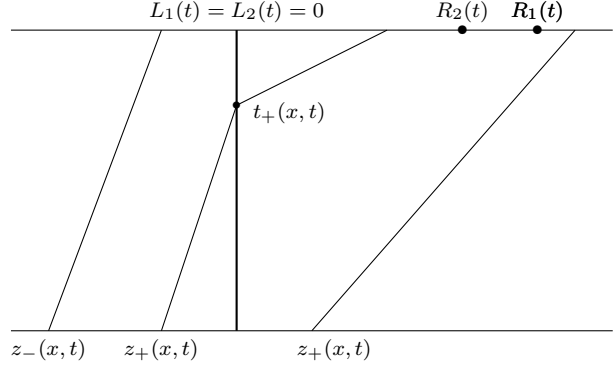
\label{figure-1}
		\begin{tikzpicture}
			\draw (-3,0)--(5,0);
			\draw[thick] (0,0)--(0,4);
			\draw (-3,4)--(5,4);
			\draw (2,4)--(0,3);
			\draw (0,3)--(-1,0);
			\draw (4.5,4)--(1,0);
			\draw (-1,4)--(-2.5,0);
			%numbering
			\draw (.7,2.9) node{\scriptsize $t_{+}(x,t)$};
			\draw (4,4.25) node{\scriptsize $R_{1}(t)$};
			\draw (0,3) node{\tiny $\bullet$};
			\draw (4,4) node{\scriptsize $\bullet$};
			\draw (3,4.25) node{\scriptsize $R_{2}(t)$};
			\draw (3,4) node{\scriptsize $\bullet$};
			\draw (4,4.25) node{\scriptsize $R_{1}(t)$};
			\draw (0,4.25) node{\scriptsize $L_{1}(t)=L_{2}(t)=0$};
			%\draw (-2,5.25) node{\scriptsize $L_{2}(t)$};
			\draw (-1,-.25) node{\scriptsize $z_{+}(x,t)$};
			\draw (-2.5,-.25) node{\scriptsize $z_{-}(x,t)$};
			\draw (1.3,-.25) node{\scriptsize $z_{+}(x,t)$};
		\end{tikzpicture}
		\captionof{figure}{An illustration of solution for \descref{thm1-case-2}{Case 2} and $L_{i}(t)$ and $R_{i}(t)$ curves}
	\end{center}
	There is a maximum principle  for such entropy solutions, but more complicate than for $f=g$,
	\begin{equation}\label{eq:max}
		\|u\|_\infty \leq \max \left(  \|u_0\|_\infty ,   \sup_{|v| \leq \|u_0\|_\infty }|f^{-1}_+ ( g(v))|,  \sup_{|v| \leq \|u_0\|_\infty }|g^{-1}_- ( f(v))|  \right)=:S_{f,g,\norm{u_0}_\f}.
	\end{equation}
\begin{center}
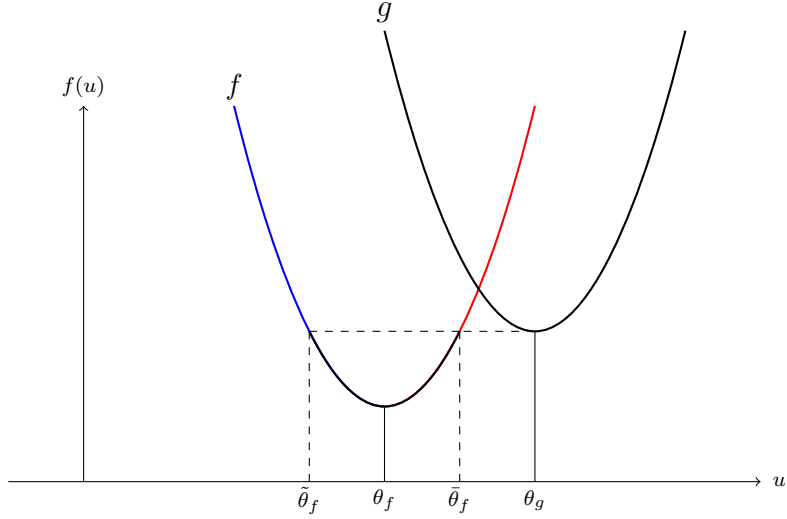
\label{fig-2}
	\begin{tikzpicture}
		%u-axis
		\draw[->] (-5,-1)--(5,-1);
		%f-axis
		\draw[->] (-4,-1)--(-4,4);
		\draw[dashed] (2,1)--(1,1);
		\draw[dashed] (1,1)--(1,-1);
		\draw[dashed] (1,1)--(-1,1);
		\draw[dashed] (-1,1)--(-1,-1);
		\draw (0,0)--(0,-1);
		\draw (2,1)--(2,-1);
		%flux f
		\draw[thick,red] (0,0) parabola (2,4);
		\draw [thick](0,0) parabola (1,1);
		\draw[thick,blue] (0,0) parabola (-2,4);
		\draw[thick](0,0) parabola (-1,1);
		%flux g
		\draw[thick] (2,1) parabola (4,5);
		\draw[thick] (2,1) parabola (0,5);
		%numbering
		\draw (-2, 4.25) node{\small $f$};
		\draw (0, 5.25) node{ $g$};
		\draw (5.25, -1) node{\scriptsize $u$};
		\draw (-4,4.25) node{\scriptsize $f(u)$};
		\draw (0, -1.25) node{\scriptsize $\theta_{f}$};
		\draw (2, -1.25) node{\scriptsize $\theta_{g}$};
		\draw (1, -1.25) node{\scriptsize $\bar{\theta}_{f}$};
		\draw (-1, -1.25) node{\scriptsize $\tilde{\theta}_{f}$};
	\end{tikzpicture}
	\captionof{figure}{Illustration of the fluxes}
\end{center}
Without loss of generality, as it is shown in the Figure \ref{fig-2}, we can assume that $\min f<\min g$ for the proofs of all main results below. This choice inforces the values of the entropy solution at the interface being outside $(\tilde{\theta}_{f},\bar{\theta}_{f})$.  Thus the function $f'$ is far from $0$ at the interface. Moreover, the function $f'^{-1}$ is Lipschitz outside $(\tilde{\theta}_{f},\bar{\theta}_{f})$. For restricted fluxes the function $f_{+}^{-1}$ is also Lipschitz outside $(\tilde{\theta}_{f},\bar{\theta}_{f})$. The Figure \ref{fig-2} illustrate that singular maps $f_{+}^{-1}g$ and $g_{-}^{-1}f$ are Lipschitz and H\"older continuous respectively which are proved in  \ref{appendix-a}, Lemma \ref{lipschitz}. 
	\section{Proof of main results}\label{mainproof}
	This long section is devoted to prove the fractional $BV$ regularity of the entropy solution depending on the degeneracy of the fluxes.   A key point is first to estimate the regularity of the traces at the interface.   In the next subsection \ref{2.3} we start to  study the fractional  regularity  in a favourable case when  the traces at the interface  are not near the critical values $\theta_f$ or $\theta_g$. Spatial $BV^{s}$ estimates for trace values issued from the interface is studied in Subsection \ref{spatial}. Moreover, only traces issuing from the initial data are considered. The crossing of the interface is studied later in Subsection \ref{proofs}.
	\subsection{Regularity when traces are far from critical values}\label{2.3}
	{We first prove fractional $BV$ estimates  when  the traces  at $x=0$ are far from the critical values $\theta_f$ or $\theta_g$.
    \begin{lemma}[Fractional $BV$ estimate for the traces of the solution]\label{Prop-partial}
	  Let $f, g$ be satisfying \eqref{fluxc} with exponents $p,q$ respectively. Let $0<a<b<\f$. Then the following holds:
	  \begin{enumerate}
	      \item If $u(0-,t)>\theta_{g}$ for a.e. $t\in(a,b)$, then we have
	     \begin{equation}\label{TV-bd-1}
					TV^{\frac{1}{q}}(u(0-,\cdot),(a,b))\leq  C_g \frac{b}{a},
				\end{equation}
				where $C_{g}>0$ is constant depending only on $g$.
			\item If $u(0+,t)<\theta_{f}$ for a.e. $t\in(a,b)$, then we have
			\begin{equation}\label{TV-bd-2}
					TV^{\frac{1}{p}}(u(0+,\cdot),(a,b))\leq C_f \frac{b}{a},
				\end{equation}
				where $C_f>0$ is a constant depending on $f$.
	  \end{enumerate}
	\end{lemma}
	\begin{proof}
			Since $u(0-,t)>\theta_{g}$ and $g' \ge 0$ 
			 on $(\theta_g, +\infty)$ the value of the left trace comes from the left. From Theorem \ref{constant}, $u(0-,t)=(g^{\p})^{-1}\left(\frac{-z_-(0-,t)}{t}\right)$ for $t\in(a,b)$ where $t\mapsto z_-(0-,t)$ is non-increasing. Since $g$ satisfies  the non-degeneracy condition \eqref{fluxc}, from Lemma \ref{lemma:Holder} $(g^{\p})^{-1}$ is a  $1/q$-H\"older function  and there exists a constant $H_{g}$ such that 
			\begin{equation*}
				\abs{u(0-,t_1)-u(0-,t_2)}\leq H_g\abs{\frac{z_-(0-,t_1)}{t_1}-\frac{z_-(0-,t_2)}{t_2}}^{\frac{1}{q}}.
			\end{equation*}
			We observe that
			\begin{equation*}
				\abs{\frac{z_-(0-,t_1)}{t_1}-\frac{z_-(0-,t_2)}{t_2}}\leq \abs{z_-(0-,t_1)}\abs{\frac{1}{t_1}-\frac{1}{t_2}}+\frac{1}{t_2}\abs{z_-(0-,t_1)-z_-(0-,t_2)}.
			\end{equation*}
			For any partition $a\leq t_1<t_2<\cdots<t_m\leq b$,
			\begin{align*}
				&\sum\limits_{j=1}^{m-1}\abs{u(0-,t_j)-u(0-,t_{j+1})}^q\\
				&\leq {H_g}^q\sum\limits_{j=1}^{m-1}\left[\abs{z_-(0-,t_j)}\abs{\frac{1}{t_j}-\frac{1}{t_{j+1}}}+\frac{1}{t_{j+1}}\abs{z_-(0-,t_j)-z_-(0-,t_{j+1})}\right]\\
				&\leq {H_g}^q\left[\abs{z_-(0-,b)}\sum\limits_{j=1}^{m-1}\abs{\frac{1}{t_j}-\frac{1}{t_{j+1}}}+\frac{1}{a}\sum\limits_{i=1}^{m-1}\abs{z_-(0-,t_j)-z_-(0-,t_{j+1})}\right]\\
				&\leq {H_g}^q\left[\frac{\abs{z_-(0-,b)}(b-a)}{ab}+\frac{\abs{z_-(0-,a)-z_-(0-,b)}}{a}\right].
			\end{align*}
			Since $\abs{z_-(0-,a)-z_-(0-,b)}\leq \abs{z_-(0-,b)}$ and $b-a\leq b$ we have, 
			\begin{equation*}
				\sum\limits_{j=1}^{m-1}\abs{u(0-,t_j)-u(0-,t_{j+1})}^q\leq 2{H_g}^q\frac{\abs{z_-(0-,b)}}{a}.
			\end{equation*}
			From finite speed of propagation we have $\abs{z(0-,b)}\leq M_gb$ where $K_{f,g,\norm{u_0}_\f}=\sup\{\abs{g^\p(v)};\,\abs{v}\leq \norm{u_0}_\f \}$.
			%where $\norm{u}_\f$  is bounded by $S_{f,g,\norm{u_0}_\f}$ by \eqref{eq:max}.
			Hence, we get a new constant $C_g$
			\begin{equation*}
				\sum\limits_{j=1}^{m-1}\abs{u(0-,t_j)-u(0-,t_{j+1})}^q\leq C_g\frac{b}{a}.%\frac{\abs{z(0-,b)}}{a}
			\end{equation*}
			This proves \eqref{TV-bd-1}. Similarly, we can prove the \eqref{TV-bd-2}. 
		\end{proof}
		Better fractional $BV$ estimates  for the traces of the solution  are available for less singular fluxes.
	\begin{lemma}[Fractional $BV$ estimate for traces away from critical values]\label{Prop-partial-2}
	    Let $r>0$ and $f, g$ be satisfying \eqref{fluxc} with exponent $p,q$ respectively. Let $0<a<b<\f$.
	    \begin{enumerate}
	    \item If $u(0-,t)\ge\theta_{g}+r$ and $g''$ vanishes only at $\theta_{g}$ \eqref{f''}, then there exists a constant $C_{g}>0$ independent of $r$ such that the following inequality holds,
	    \begin{eqnarray}\label{TV-bd-1.1}
	        TV(u(0-,\cdot),(a,b))\le\frac{C_{g}}{\min\{g''(v)| v\in[\theta_{g}+r,||u_{0}||_{\f}]\}}\frac{b}{a}.
	    \end{eqnarray}
	    \item If $u(0+,t)\le\theta_{f}+r$ and $f''$ vanishes only at $\theta_{g}$ \eqref{f''}, then there exists a constant $C_{f}>0$ independent of $r$ such that the following inequality holds,
	    \begin{eqnarray}\label{TV-bd-2.1}
	        TV(u(0+,\cdot),(a,b))\le\frac{C_{f}}{\min\{f''(v)| v\in[-||u_{0}||_{\f}, \theta_{f}-r]\}}\frac{b}{a}.
	    \end{eqnarray}
	    \end{enumerate}
	\end{lemma}
 Lemma \ref{Prop-partial-2} will be used later with constant $r$ given by either $\theta_{f}-\tilde{\theta}_{f}$ or $\bar{\theta}_{f}-{\theta}_{f}$ as shown in Figure \ref{fig-2}. The fact that $r$ is  a positive constant is crucial to get uniform estimates later.
	\begin{proof}
	    For any $x, y\in\R$ consider
			\begin{eqnarray}\nonumber
	       \abs{x-y} =\abs{g'(g'^{-1}(x)-g'(g'^{-1}(y)}
	       %\\\nonumber
	       = g''(\xi)\abs{g'^{-1}(x)-g'^{-1}(y)},
			\end{eqnarray}
			where $\xi\in(x, y)$.
			Now for \eqref{TV-bd-1.1}, Theorem \ref{constant} gives, $u(0-,t)=(g^{\p})^{-1}\left(\frac{-z_-(0-,t)}{t}\right)$ for $t\in(a,b)$ where $t\mapsto z_-(0-,t)$ is non-increasing. Thus,
			\begin{eqnarray}\nonumber
			   \abs{u(0-,t_{1})-u(0-,t_{2)}}&=&\abs{g'^{-1}\left(\frac{-z_-(0-,t_{1})}{t_{1}}\right)-g'^{-1}\left(\frac{-z_-(0-,t_{2})}{t_{2}}\right)}\\\nonumber
			   &\le&\min\{g''(v);v\in[\theta_g+r,\norm{u_0}_\f]\}^{-1}\abs{\frac{z_-(0-,t_{1})}{t_{1}}-\frac{z_-(0-,t_{2})}{t_{2}}}.
			\end{eqnarray}
			Now the similar calculation as to prove \eqref{TV-bd-1} gives  \eqref{TV-bd-1.1}.  
		By similar arguments \eqref{TV-bd-2.1} can be proven for $f$.
	\end{proof}
	\subsection{Spatial $BV^s$ estimates for values originating from the interface}\label{spatial}
	Now, far from the interface and restricted flux, when the values of the solution are far from the critical values of $f$ and $g$, a  $BV$ estimate is available.  
	The following inequality are also valid in $BV^{s}$ for free and used later with other $BV^s$ estimates.
		\begin{lemma}[$BV$ and $BV^{s}$ estimates for the solution]\label{lemma:partial-1}
			Let $u$ be an entropy solution and $R_1(t)>0$ for some fixed  $t>0$. Let $0<a<b<R_1(t)$ and $S_{f,g,\norm{u_0}_\f}$ be as in \eqref{eq:max}. Let $r>0$, $f$ satisfies \eqref{fluxc} and $f^{\p\p}$ vanishes only on $\theta_f$ \eqref{f''}. If  $u(x,t)\geq \theta_f+r$ for $a\le x \le b$, then there exists a constant $C_{f,g,\norm{u_0}_\f}>0$ such that
				\begin{equation}\label{TVs-4.2-1}
			TV^{s}(u(\cdot,t),[a,b])\leq \frac{C_{f,g,\norm{u_0}_\f}}{\min\{f^{\p\p}(v);v\in[\theta_{f}+r,S_{f,g,\norm{u_0}_\f}]\}^{\frac{1}{s}}}\left(\frac{t-t_+(b,t)}{t-t_+(a,t)}\right)^{\frac{1}{s}},
				\end{equation}
				for all $0<s\leq 1$. 
				\end{lemma}
			The same result holds for the left side of the interface as follows:
				\begin{lemma}[$BV$ and $BV^{s}$ estimate for the solution]\label{lemma:partial-2}
				Let $u$ be an entropy solution and $L_1(t)<0$ for some $t>0$. Let $L_1(t)<a<b<0$ and $S_{f,g,\norm{u_0}_\f}$ be as in \eqref{eq:max}. Let $r>0$, flux $g$ satisfies \eqref{fluxc} and $g^{\p\p}$ vanishes only on $\theta_g$. If $u(x,t)\leq \theta_g-r$ for $a\le x \le b$, then there exists a constant $C_{f,g,\norm{u_0}_\f}>0$ such that
				\begin{equation*}
					TV^{s}(u(\cdot,t),[a,b])\leq \frac{C_{f,g,\norm{u_0}_\f}}{\min\{g^{\p\p}(v);v\in[-S_{f,g,\norm{u_0}_\f}, \theta_{g}-r]\}^{\frac{1}{s}}}\left(\frac{t-t_-(b,t)}{t-t_-(a,t)}\right)^{\frac{1}{s}},
				\end{equation*}
				for all $0<s\le1$.
			\end{lemma}
		\begin{proof}
			 Theorem \ref{constant} gives,
			\begin{equation*}
				u(x,t)=(f^{\p})^{-1}\left(\frac{x}{t-t_+(x,t)}\right)\mbox{ for }x\in(0,R_1(t)).
			\end{equation*}
			Fix a partition $a\leq x_1<x_2<\cdots<x_m\leq b$. Then, as in the proof of inequality  \eqref{TV-bd-1.1}, it follows, 
			\begin{align*}
				\sum\limits_{j=1}^{m-1}\abs{u(x_j,t)-u(x_{j+1},t)}^{\frac{1}{s}}&=\sum\limits_{j=1}^{m-1}\abs{(f^{\p})^{-1}\left(\frac{x_j}{t-t_+(x_j,t)}\right)-(f^{\p})^{-1}\left(\frac{x_{j+1}}{t-t_+(x_{j+1},t)}\right)}^{\frac{1}{s}}\\
				&\leq \frac{1}{\min\{f^{\p\p}(v);v\in[\theta_f+r,S_{f,g,\norm{u_0}_\f}]\}^{\frac{1}{s}}}\sum\limits_{j=1}^{m-1}\abs{\frac{x_j}{t-t_+(x_j,t)}-\frac{x_{j+1}}{t-t_+(x_{j+1},t)}}^{\frac{1}{s}}.
			\end{align*}
			We calculate
			\begin{align*}
				\abs{\frac{x_j}{t-t_+(x_j,t)}-\frac{x_{j+1}}{t-t_+(x_{j+1},t)}}&\leq \abs{x_j}\abs{\frac{1}{t-t_+(x_j,t)}-\frac{1}{t-t_+(x_{j+1},t)}}+\frac{1}{t-t_+(x_{j+1},t)}\abs{x_j-x_{j+1}}\\
				&\leq b\abs{\frac{1}{t-t_+(x_j,t)}-\frac{1}{t-t_+(x_{j+1},t)}}+\frac{1}{t-t_+(a,t)}\abs{x_j-x_{j+1}}.
			\end{align*}
			Hence, by the convexity yields, \textcolor{black}{$(a+b)^{\frac{1}{s}}\le 2^{\frac{1-s}{s}}\left(a^{\frac{1}{s}}+b^{\frac{1}{s}}\right) $} \mbox {and we get}
			\begin{align*} 
				&\sum\limits_{j=1}^{m-1}\abs{\frac{x_j}{t-t_+(x_j,t)}-\frac{x_{j+1}}{t-t_+(x_{j+1},t)}}^{\frac{1}{s}}\\
				&\leq\frac{1}{2^{\frac{s-1}{s}}}\left( \sum\limits_{j=1}^{m-1}b^{\frac{1}{s}}\abs{\frac{1}{t-t_+(x_j,t)}-\frac{1}{t-t_+(x_{j+1},t)}}^{\frac{1}{s}}+\sum\limits_{j=1}^{m-1}\frac{1}{(t-t_+(a,t))^{\frac{1}{s}}}\abs{x_j-x_{j+1}}^{\frac{1}{s}}\right)\\
				&\leq\frac{1}{2^{\frac{s-1}{s}}}\left( b^{\frac{1}{s}}\abs{\frac{1}{t-t_+(a,t)}-\frac{1}{t-t_+(b,t)}}^{\frac{1}{s}}+\left(\frac{b-a}{t-t_+(a,t)}\right)^{\frac{1}{s}}\right)\\
				&\leq 2^\frac{1}{s}\left(\frac{b}{t-t_+(a,t)}\right)^{\frac{1}{s}}.
			\end{align*}
			In the last step we have used $b-a\leq b$ and $(t-t_+(b,t))-(t-t_+(a,t))\leq t-t_+(b,t)$. Note that $b\leq K_{f,g,\norm{u_0}_\f} (t-t_+(b,t))$ where $K_{f,g,\norm{u_0}_\f}=\sup\{\abs{f^{\p}};\abs{v}\leq S_{f,g,\norm{u_0}_\f}\}$ where $S_{f,g,\norm{u_0}_\f}$ is defined as in \eqref{eq:max}. 
		\end{proof}
		The following lemma deals with spatial regularity of the entropy solution for the right side of the interface. Inequality \eqref{lemma:partial-1.2} does not used the restricted non-degeneracy condition.
		\begin{lemma}\label{lemma:partial-1.2}
		    Let $u$ be an entropy solution and $R_1(t)>0$ for some fixed  $t>0$. Let $0<a<b<R_1(t)$ and $S_{f,g,\norm{u_0}_\f}$ be as in \eqref{eq:max}. If $f$ only satisfies \eqref{fluxc} with exponent $p$ then we have
		    \begin{equation}\label{TV-4.2-2}
	       TV^{\frac{1}{p}}(u(\cdot,t),[a,b])\leq C_{f,g,\norm{u_0}_\f}\frac{t-t_+(b,t)}{t-t_+(a,t)}.
				\end{equation}
		\end{lemma}
	The same result holds for the left side of the interface as follows.
		\begin{lemma}\label{lemma:partial-3}
		Let $u$ be an entropy solution and $L_1(t)<0$ for some $t>0$. Let $L_1(t)<a<b<0$. If $g$ satisfies \eqref{fluxc} with exponent $q$ then we have
		\begin{equation}\label{TVs-4.3-2}
			TV^{\frac{1}{q}}(u(\cdot,t),[a,b])\leq C_{f,g,\norm{u_0}_\f}\frac{t-t_-(b,t)}{t-t_-(a,t)}.
		\end{equation}
	\end{lemma}
		By a similar argument as previous Lemma \ref{lemma:partial-1}  the inequality \eqref{TV-4.2-2} of Lemma \ref{lemma:partial-1.2} can be proven, so it is not written here. 
%		Lemma \ref{lemma:partial-2} is much the same to Lemma \ref{lemma:partial-1} except the fact that it deals with regularity of solution for left side of interface.
%		\begin{lemma}[$BV$ and $BV^{s}$ estimate for the solution]\label{lemma:partial-2}
%			Let $u$ be an entropy solution and $L_1(t)<0$ for some $t>0$. Let $L_1(t)<a<b<0$ and $S_{f,g,\norm{u_0}_\f}$ be as in \eqref{eq:max}. Let $r>0$, flux $g$ satisfies \eqref{fluxc} and $g^{\p\p}$ vanishes only on $\theta_g$. If $u(x,t)\leq \theta_g-r$ for $a\le x \le b$, then there exists a constant $C_{f,g,\norm{u_0}_\f}>0$ such that
%				\begin{equation*}
%					TV^{s}(u(\cdot,t),[a,b])\leq \frac{C_{f,g,\norm{u_0}_\f}}{\min\{g^{\p\p}(v);v\in[-S_{f,g,\norm{u_0}_\f}, \theta_{g}-r]\}^{\frac{1}{s}}}\left(\frac{t-t_-(b,t)}{t-t_-(a,t)}\right)^{\frac{1}{s}},
%				\end{equation*}
%				for all $0<s\le1$.
%			\end{lemma}
	%	Following the same line of ideas of Lemma \ref{lemma:partial-1} can be proven is not written here. 
		%	Lemma \ref{lemma:partial-3} also similar to Lemma \ref{lemma:partial-1.2} deals with regularity of the entropy solution for left side of the interface. 
	
	}
	%-------------------------------------------------------------------------------------------------------------------------------------
	%-------------------------------------------------------------------------------------------------------------------------------------
	\subsection{Smoothing effect for restricted nonlinear fluxes}\label{proofs}
	\textcolor{black}{Now we are ready to prove Theorem  \ref{fnoteqg}.} To end this, an \textcolor{black}{arbitrary} partition  is fixed and divided in several parts. Some are far from the interface and the generalized variation is estimated with a regularizing effect for a scalar conservation laws without a boundary. Some others are  near the interface where the Lax-Oleinik formula for the solution \cite{Kyoto} is used with previous lemmas.
	\begin{proof}[Proof of Theorem \ref{fnoteqg}:]
		Since $f(\theta_{f})\not=g(\theta_{g})$,  without loss of generality assume that $f(\theta_{f})<g(\theta_{g})$ as in Figure \ref{fig-2}. It is enough to consider the following two cases, the other cases are similar.\\
		Case(i): $L_{1}(t)=0$ and $R_{1}(t)\ge0$.
		
		 Consider an arbitrary partition $\{-M=x_{-n}<\cdots<x_{-1}<x_{0}\le0<x_{1}<\cdots<x_{l}\le R_{2}(t)<x_{l+1}<\cdots<x_{m}\le R_{1}(t)<x_{m+1}<\cdots<x_{n}=M \}$. Then,
		\begin{multline*}
			\sum_{i=-n}^{n-1}|u(x_{i},t)-u(x_{i+1},t)|^{1/s}=\sum_{i=-n}^{-1}|u(x_{i},t)-u(x_{i+1},t)|^{1/s}+\sum_{i=m+1}^{n-1}|u(x_{i},t)-u(x_{i+1},t)|^{1/s}\\
			+\sum_{i=1}^{l-1}|u(x_{i},t)-u(x_{i+1},t)|^{1/s}+\sum_{i=l+1}^{m-1}|u(x_{i},t)-u(x_{i+1},t)|^{1/s}\\
			+|u(x_{0},t)-u(x_{1},t)|^{1/s}+|u(x_{l},t)-u(x_{l+1},t)|^{1/s}+|u(x_{m},t)-u(x_{m+1},t)|^{1/s}.
		\end{multline*}
		From Theorem \ref{constant}, solution $u$ is constant between $R_{2}(t)$ to $R_{1}(t)$ which means variation is zero for this interval. Now from the Lax-Oleinik formula in Theorem \ref{constant} and bounding the last three terms yield,
		\begin{multline*}
			\sum_{i=-n}^{n}|u(x_{i},t)-u(x_{i+1},t)|^{1/s}\le\underbrace{\sum_{i=-n}^{-1}|u(x_{i},t)-u(x_{i+1},t)|^{1/s}}_\text{I}+\underbrace{\sum_{i=m+1}^{n-1}|u(x_{i},t)-u(x_{i+1},t)|^{1/s}}_\text{III}\\
			+\underbrace{\sum_{i=1}^{l-1}|f^{-1}_{+}g(u_{0}(z_{+}(x_{i},t)))-f^{-1}_{+}g(u_{0}(z_{+}(x_{i+1},t)))|^{1/s}}_\text{II}+3(2||u_{0}||_{\f})^{1/s}.
		\end{multline*}
		\textcolor{black}{	Now we wish to estimate the terms \RN{1}, \RN{2}, and \RN{3}.  The simplest terms  \RN{1},  \RN{3} are estimated as in \cite{junca1,CJLO}.  First taking the \RN{1} into the account.  Since $f$ and $g$ are satisfying the flux non-degeneracy condition \eqref{fluxc}, by Lemma \ref{lemma:Holder}, the maps $u\mapsto (g')^{-1}(u)$ and $u\mapsto (f')^{-1}(u)$ are H\"older continuous with exponents $q^{-1}$ and $p^{-1}$ respectively. 
			From Theorem \ref{constant}, 
			\begin{eqnarray*}
			 u(x, t)=(g')^{-1}\left(\frac{x-z_{-}(x, t)}{t}\right), &\mbox{ for } x<0, 
			\end{eqnarray*}
			  then for $-M\leq x_i<x_{i+1}\leq 0$, from Lemma \ref{lemma:Holder}
			\begin{eqnarray*}
				|u(x_{i},t)-u(x_{i+1},t)|^{q}&=&\left|(g')^{-1}\left(\frac{x_{i}-z_{-}(x_{i},t)}{t}\right)-(g')^{-1}\bigg(\frac{x_{i+1}-z_{-}(x_{i+1},t)}{t}\bigg) \right|^{q}\\
				&\le&\bigg(C_{2}^{-q^{-1}}\bigg|\frac{x_{i}-z_{-}(x_{i},t)}{t}-\frac{x_{i+1}-z_{-}(x_{i+1},t)}{t}\bigg|^{q^{-1}} \bigg)^{q},
			\end{eqnarray*}
			using triangle inequality we obtain,
			\begin{equation*}
				|u(x_{i},t)-u(x_{i+1},t)|^{q}\le C_{2}^{-1}\bigg|\frac{x_{i}-x_{i+1}}{t}\bigg|+C_{2}^{-1}\bigg|\frac{z_{-}(x_{i},t)-z_{-}(x_{i+1},t)}{t}\bigg|.
			\end{equation*}
			Since $|x_i|, |x_{i+1}|\le M$ and $x=z_{-}(x, t)+g'(u(x,0))t$ hence, we get 
			\begin{equation}\label{calc-1}
				TV^{q^{-1}}u(\sigma\cap[-M,0])\le \frac{4M}{C_{2} t}+\frac{1}{C_2}\sup\left\{\abs{g^{\p}(v)};\,\abs{v}\leq \norm{u_0}_{L^{\f}(\R)}\right\}.
			\end{equation}
			In similar fashion, for the term $\RN{3}$ we have,  
			\begin{equation}\label{estimation-III}
				TV^{p^{-1}}u(\sigma\cap[R_1(t),M])\le \frac{4M}{C_{1} t}+\frac{1}{C_1}\sup\left\{\abs{f^{\p}(v)};\,\abs{v}\leq \norm{u_0}_{L^{\f}(\R)}\right\}.%TV^{p^{-1}}u(\sigma\cap(R_{1}(t),\f))\le \frac{4M}{C_{1} t}+t\sup_{u}|f'|.
			\end{equation}
			Now we will estimate the $\RN{2}$ term. From the definition of $s$, $s\le1/p$ and $s\le1/(q+1)$. 
			Rest of the proof for this case is divided into two sub-cases.}
		{\begin{enumerate}
				\item Consider the situation when $t_+^{{min}}(t)=\inf\{t_+(x,t);x\in(0,R_1(t))\}\geq t/2$. The fact $t_{+}^{min}>t/2>0$ implies that the characteristics reaching the left side of the interface at $(0-,t_+)$  has a  positive speed, hence $u(0-,t_{+}(x,t))>\theta_{g}$ for all $x\in(0,R_{1}(t))$ (Figure \ref{fig-2}). Therefore, the inequality \eqref{TV-bd-1} of Lemma \ref{Prop-partial} gives $TV^{\frac{1}{q}}(u(0-,\cdot)(t_+^{min},t))\leq C_g  \frac{t}{t/2} = 2 C_g$. Since $s\le\frac{1}{q+1}<\frac{1}{q}$, Lemma \ref{TVs_norm} yields $TV^{s}(u(0-,\cdot)(t_+^{min},t))\leq \mathcal{O}(1)$ and then $\RN{2}\leq \mathcal{O}(1) $.
				
				\item Next focus on the sub-case when $t_+^{{min}}(t)=\inf\{t_+(x,t);x\in(0,R_1(t))\}< t/2$. As previous subcase we already have $TV^{s}(u(0-,\cdot)(t/2,t))\leq 2C_g$. Let $j_0>0$ such that $t_+(x_j,t)\geq t/2$ for $0<j\leq j_0 $ and $t_+(x_j,t)<t/2$ for $j_0<j\leq l-1$. Since $u(x_j,t)=u(0+,t_+(x_j,t))=f_{+}^{-1}g(u(0-,t_{+}(x_{j},t))$ for $0<j<l-1$, from Lemma \ref{lipschitz}, $f_{+}^{-1}g$ is Lipschitz function, hence
				\begin{equation*}
					\sum\limits_{j=1}^{j_0}\abs{u(x_j,t)-u(x_{j-1},t)}^{\frac{1}{s}}\leq \mathcal{O}(1).
				\end{equation*}
				Let $\bar{\theta}_f> \theta_f$ be such that $f(\bar{\theta}_f)=g(\theta_g)$ as shown in Figure \ref{fig-2}. Then by RH condition \eqref{RH} observe that $u(x_j,t)\geq \bar{\theta}_f$. From the inequality \eqref{TVs-4.2-1} of Lemma \ref{lemma:partial-1} we get
				\begin{equation*}
					\sum\limits_{j=j_0+1}^{l-2}\abs{u(x_j,t)-u(x_{j+1},t)}^{\frac{1}{s}}\leq \mathcal{O}(1).
				\end{equation*}
		 
				Subsequently, we get
				\begin{equation}\label{est-II-1}
					\RN{2}\leq \mathcal{O}(1).
				\end{equation}
			\end{enumerate}
		}
		Hence combining the estimates on \RN{1}, \RN{2} and \RN{3} for constant $C_{f,g,\norm{u_0}_\f}>0$ we have 
			\begin{equation*}
				\sum_{i=-n}^{n}|u(x_{i},t)-u(x_{i+1},t)|^{1/s}\le C_{f,g,\norm{u_0}_\f} \left(  1 + \frac{1}{t} \right) .
			\end{equation*}
			\descitem{Case (ii):}{case2-f-n-g}  $R_{1}(t)=0$, $L_{1}(t)<0$. Unlike previous case, this case is not as good due to the fact that $g_-^{-1}f$ is only H\"older continuous  and not Lipschitz. 
			Let us consider the partition $\sigma=\{-M=x_{-n}<\cdots<x_{m}\le L_{2}(t)=L_{1}(t)<x_{m+1}<\cdots<x_{0}\le R_{2}(t)=R_{1}(t)=0<x_{1}<\cdots\le x_{n}=M \}$ and then
			\begin{eqnarray*}
				\sum_{i=-n}^{n}|u(x_{i},t)-u(x_{i+1},t)|^{1/s}&=&\sum_{i=-n}^{m-1}|u(x_{i},t)-u(x_{i+1},t)|^{1/s}+\sum_{i=1}^{n}|u(x_{i},t)-u(x_{i+1},t)|^{1/s}\\
				&+&\sum_{i=m+1}^{-1}|u(x_{i},t)-u(x_{i+1},t)|^{1/s}+|u(x_{0},t)-u(x_{1},t)|^{1/s}\\
				&+&|u(x_{m},t)-u(x_{m+1},t)|^{1/s}.
			\end{eqnarray*}
			From Theorem \ref{constant} we get, 
			\begin{eqnarray*}
				\sum_{i=-\f}^{\f}|u(x_{i},t)-u(x_{i+1},t)|^{1/s}&=&\underbrace{\sum_{i=-n}^{m-1}|u(x_{i},t)-u(x_{i+1},t)|^{1/s}}_\text{I}+2(2||u_{0}||_{\f})^{1/s}\\
				&+&\underbrace{\sum_{i=m+1}^{-1}|g^{-1}_{-}(f(u_{0}(z_{-}(x_{i},t))))-g^{-1}_{-}(f(u_{0}(z_{-}(x_{i+1},t))))|^{1/s}}_\text{II}\\
				&+&\underbrace{\sum_{i=1}^{n}|u(x_{i},t)-u(x_{i+1},t)|^{1/s}}_\text{III}.\\
			\end{eqnarray*}
			Similarly to \descref{case1-f-n-g}{Case (i)} we bound \RN{1}, \RN{3} as in \eqref{calc-1}, \eqref{estimation-III} to get
			\begin{equation*}
				\RN{1}+ \RN{3}\le \frac{C_{f,g,||u_{0}||_{\f}}M}{t}.
		\end{equation*}
		 Now the term \RN{2} consider as previous term \RN{2} and divide into two sub-cases.
			\begin{enumerate}
				\item We first consider the situation when $t_-^{{min}}(t)=\inf\{t_-(x,t);x\in(L_1(t),0)\}\geq t/2$. The RH condition \eqref{RH} implies that $u(0+,\cdot)\le \tilde{\theta}_{f}$, see Figure \ref{fig-2}, the inequality \eqref{TV-bd-2.1} of Lemma \ref{Prop-partial-2} gives
			\begin{equation}	TV(u(0+,\cdot)(t_-^{min},t))\leq C_{f,g,\norm{u_0}_\f}.
			\end{equation}
			Note that $g_-^{-1}\circ f$ is H\"older continuous function with exponent $ \frac{1}{q+1}$. Hence we have
				\begin{equation}
					\RN{2}=\sum\limits_{j=m+1}^{-1}\abs{u(x_j,t)-u(x_{j+1},t)}^{\frac{1}{q+1}}\leq C_{f,g,\norm{u_0}_\f}.
				\end{equation} 
				
				\item Next we focus on the sub-case when $t_-^{{min}}(t)=\inf\{t_-(x,t);x\in(L_1(t),0)\}< t/2$. Let $j_0<0$ such that $t_+(x_j,t)\geq t/2$ for $ j_0\leq j<0 $ and $t_+(x_j,t)<t/2$ for ${m+1}<j<j_0$. In previous sub-case we have 
				\begin{equation}
					\sum\limits_{j=j_0}^{-1}\abs{u(x_j,t)-u(x_{j+1},t)}^{\frac{1}{q+1}}\leq C_{f,g,\norm{u_0}_\f}.
				\end{equation}
				Note that for $m+1<j<j_0$, $u(x_j,t)=u(0-,t_-(x_j,t))\leq \theta_g$. From the inequality \eqref{TVs-4.3-2} of Lemma \ref{lemma:partial-3} we have 
				\begin{equation}
					\sum\limits_{j=m+1}^{j_0-1}\abs{u(x_j,t)-u(x_{j+1},t)}^{\frac{1}{q}}\leq C_{f,g,\norm{u_0}_\f}.
				\end{equation}
				Subsequently, we get
				\begin{equation}\label{est-II-2}
					\RN{2}\leq C_{f,g,\norm{u_0}_\f}+\norm{2u_0}_{\f}^{\frac{1}{q+1}}.%+\frac{C_{f,g}}{\min\{g^{\p\p}(v),v\in[-\norm{u_0}_\f,\underline{u}]\}^{\frac{1}{s_1}}}.
				\end{equation}
			\end{enumerate} 	
			Hence, from the estimates on \RN{1}, \RN{2} and \RN{3} we get
			\begin{equation}\label{eq:estimate}
				\sum_{i=-n}^{n}|u(x_{i},t)-u(x_{i+1},t)|^{\frac{1}{q+1}}\leq C_{f,g,\norm{u_0}_\f}+3(2||u_{0}||_{\f})^{\frac{1}{q+1}}+\frac{C_{f,g}M}{t}.
			\end{equation}	
		
	\end{proof} 
	\subsection{Generalization for $BV^s$ initial data}
	Now we are able to prove  Theorem \ref{th:f-not-g}. For this, again we divide the domain in several parts. Here initial data belongs to $BV^{s}$. If $s$ is very small then far from the interface estimates comes from the regularizing effect. If $s$ is near to $1$ then outside interface initial data regularity propagates. For the estimate on the solution near interface again we use Lax-Oleinik formula from \cite{Kyoto}. 
	\begin{proof}[Proof of Theorem \ref{th:f-not-g}]
		Since $f(\theta_{f})\not=g(\theta_{g})$, without loss of generality we assume that $f(\theta_{f})<g(\theta_{g})$, see Figure \ref{fig-2} because other case can be done in a similar way. Hence, from Theorem \ref{constant} we have $L_{2}(t)=L_{1}(t)$ then it is enough to consider the following two cases.
		\begin{description}
			\descitem{Case (i):}{case1-f-n-g} If $L_{1}(t)=0$ and $R_{1}(t)\ge 0$.\\
			Consider the partition $\sigma=\{-M=x_{-n}\le\cdots<x_{-1}<x_{0}\le0<x_{1}<\cdots<x_{l}\le R_{2}(t)<x_{l+1}<\cdots<x_{m}\le R_{1}(t)<x_{m+1}<\cdots\le x_{n}=M \}$ and %$\{\cdots<x_{-1}<x_{0}\le0<x_{1}<\cdots<x_{l}\le R_{2}(t)<x_{l+1}<\cdots<x_{m}\le R_{1}(t)<x_{m+1}<\cdots \}$ and fix 
			$$s_{1}=\min\{\gamma , \max\{\nu, s\}\}\in(0,1).$$
			Then
			\begin{eqnarray*}
				&&\sum_{i=-n}^{n-1}|u(x_{i},t)-u(x_{i+1},t)|^{1/s_{1}}\\
				&=&\sum_{i=-n}^{-1}|u(x_{i},t)-u(x_{i+1},t)|^{1/s_{1}}+\sum_{i=m+1}^{n-1}|u(x_{i},t)-u(x_{i+1},t)|^{1/s_{1}}\\
				&+&\sum_{i=1}^{l-1}|u(x_{i},t)-u(x_{i+1},t)|^{1/s_{1}}+\sum_{i=l+1}^{m-1}|u(x_{i},t)-u(x_{i+1},t)|^{1/s_{1}}\\
				&+&|u(x_{0},t)-u(x_{1},t)|^{1/s_{1}}+|u(x_{l},t)-u(x_{l+1},t)|^{1/s_{1}}\\
				&+&|u(x_{m},t)-u(x_{m+1},t)|^{1/s_{1}}.
			\end{eqnarray*}
			From Theorem \ref{constant}, the entropy solution is constant between $R_{2}(t)$ and $R_{1}(t)$ which means variation is zero for this interval. Hence, %Now again from Theorem \ref{constant} we get,
			\begin{align*}
				\sum_{i=-n}^{n-1}|u(x_{i},t)-u(x_{i+1},t)|^{1/s_{1}}&=\underbrace{\sum_{i=-n}^{-1}|u(x_{i},t)-u(x_{i+1},t)|^{1/s_{1}}}_\text{I}+3(2||u_{0}||_{\f})^{1/s_{1}}\\
				&+\underbrace{\sum_{i=1}^{l-1}|f^{-1}_{+}g(u_{0}(z_{+}(x_{i},t)))-f^{-1}_{+}g(u_{0}(z_{+}(x_{i+1},t)))|^{1/s_{1}}}_\text{II}\\
				&+\underbrace{\sum_{i=m+1}^{n-1}|u(x_{i},t)-u(x_{i+1},t)|^{1/s_{1}}}_\text{III}.\\
			\end{align*}
			From the choice of $s_1$, we get $s_1\leq \max\{s,1/q\}$. If $1/q>s$, then $s_1<1/q$. By a similar argument as in \eqref{calc-1} we have 
			\begin{equation*}
				\sum_{i=-n}^{-1}|u(x_{i},t)-u(x_{i+1},t)|^{1/s_{1}}\leq \frac{4M}{C_{2} t}+\frac{1}{C_2}\sup\left\{\abs{g^{\p}(v)};\,\abs{v}\leq \norm{u}_{L^{\f}(\R\times[0,T])}\right\}.
			\end{equation*}
			If $s>1/q$ then $s_1<s$ and we use the regularity of initial data to estimate $\RN{1}$ so from Lemma \ref{TVs_norm} $\RN{1}\le D\cdot TV^{s}(u_{0})$. Combining both the estimate we can write %By a similar argument as in the proof of  Theorem \ref{bvs1} 
			\begin{equation}\label{est-I}
				\RN{1}\leq TV^{s}(u_{0})+\frac{4M}{C_{2} t}+\frac{1}{C_2}\sup\left\{\abs{g^{\p}(v)};\,\abs{v}\leq \norm{u}_{L^{\f}(\R\times[0,T])}\right\}.
			\end{equation}
			Similarly we have  %For \RN{1} and \RN{3}, similarly as the proof of Theorem \ref{linf},
			\begin{equation}\label{est-III}
				\RN{3}\leq TV^{s}(u_{0})+\frac{4M}{C_{1} t}+\frac{1}{C_1}\sup\left\{\abs{f^{\p}(v)};\,\abs{v}\leq \norm{u}_{L^{\f}(\R\times[0,T])}\right\}.
			\end{equation}
			From Lemma \ref{lipschitz} we know that $f_{+}^{-1}g(\cdot)$ is a Lipschitz continuous. Hence, the term \RN{2} can be estimated as
			\begin{eqnarray*}
				\RN{2}&=&\sum_{i=1}^{l-1}|f^{-1}_{+}g(u_{0}(z_{+}(x_{i},t)))-f^{-1}_{+}g(u_{0}(z_{+}(x_{i+1},t)))|^{1/s_{1}}\\
				&\le&C\cdot \sum_{i=1}^{l-1}|u_{0}(z_{+}(x_{i},t))-u_{0}(z_{+}(x_{i+1},t))|^{1/s_{1}}.
			\end{eqnarray*}
			If $s>1/q$ then we have $s_1<s$, from Lemma \ref{TVs_norm}, $\RN{2}\leq D\cdot TV^{s}(u_0)$. For the case $s<1/q$ we do not know whether $s_1<s$ holds or not but we surely have $s_1<1/q$. For this case we use the regularizing  effect for solutions of conservation laws due to non-degeneracy of $g$ \cite{junca1}. For term $\RN{2}$ we note that the estimate \eqref{est-II-1} in proof of Theorem \ref{fnoteqg}. Hence combining the estimates on \RN{1}, \RN{2} and \RN{3} we get 
			Hence, from the estimates on \RN{1}, \RN{2} and \RN{3} we get 
			\begin{equation*}
				\sum_{i=-n}^{n-1}|u(x_{i},t)-u(x_{i+1},t)|^{1/s_{1}}\le D\cdot TV^{s}(u_{0})+3(2||u_{0}||_{\f})^{1/s_{1}}+\frac{C_{f,g}M}{t}.%\bar{C}(f,g,M,t).
			\end{equation*}
			%-----------------------------------------------------------------------------------------------------------------------------------------------------------
			\textbf{Case (ii)}: $R_{1}(t)=0$, $L_{1}(t)<0$. \\
			This case can be handled in a similar fashion as in previous case.
			
			Only difference is the estimation of $\RN{2}$ which can be done same as in \eqref{est-II-2}.% The only difference for this is that in place of $f_{+}^{-1}g(\cdot)$ we have $g_{-}^{-1}f(\cdot)$ in \RN{2} which will be H\"older continuous function with exponent $(q+1)^{-1}$.
		\end{description}
		Hence, we have proven that $u(\cdot,t)\in BV^{s_1}(-M,M)$. To show that $u(\cdot,t)\in BV^{s_1}(\R)$, we consider a partition $-\f<x_{-n}<\cdots <x_{n}<\f$ which is not necessarily contained in $[-M,M]$. We can choose $M=t\sup\{\abs{f^{\p}(v)},\abs{g^\p(v)};\abs{v}\leq \norm{u_0}_\f\}$. Suppose $\abs{x_j}\leq M$ for $-m_1\leq j\leq m_2$ for some $0<m_1,m_2\leq n$. From \eqref{eq:estimate} we get
		\begin{equation*}
			\sum_{i=-m_1}^{m_2}|u(x_{i},t)-u(x_{i+1},t)|^{\frac{1}{q+1}}\leq C_{f,g}+2(2||u_{0}||_{\f})^{\frac{1}{q+1}}.
		\end{equation*}	
		From the choice of $M$, we can see that $R_1(t)\leq M, L_1(t)\geq-M$. Hence for $i\leq -m_1$, $u(x_i,t)=u_0(z_-(x_i,t))$ and for $i\geq m_2 $, $u(x_i,t)=u_0(z_+(x_i,t))$. Subsequently,
		\begin{equation*}
			\sum_{i=-n}^{-m_1-2}|u(x_{i},t)-u(x_{i+1},t)|^{\frac{1}{s}}+\sum_{i=m_2+1}^{n-1}|u(x_{i},t)-u(x_{i+1},t)|^{\frac{1}{s}}\leq TV^s(u_0).
		\end{equation*}
		Therefore, we obtain
		\begin{equation*}
			\sum_{i=-n}^{n-1}|u(x_{i},t)-u(x_{i+1},t)|^{\frac{1}{q+1}}\leq C_{f,g,\norm{u_0}_\f}+4(2||u_{0}||_{\f})^{\frac{1}{q+1}}+TV^s(u_0).
		\end{equation*}	
		This completes the proof of Theorem \ref{th:f-not-g}.
	\end{proof}
	\subsection{Non restricted fluxes}
	Now we assume the weaker non-degeneracy on flux  the estimates on solution near interface so Lemma \ref{Prop-partial-2}, \ref{spatial} and \ref{lemma:partial-1} can not be used here. So the regularity is weaker here. 
	{\begin{proof}[Proof of Theorem \ref{theorem:gen:reg}:]
			Fix a time $t>0$. We only show for the case when $R_1(t)>0$. Note that in this case $L_1(t)=L_2(t)=0$. Suppose $t_0=\lim\limits_{x\rr R_1(t)-}t_+(x,t)$. First consider $t_0>t/2$. From Lemma \ref{Prop-partial}, we have
			\begin{equation*}
				TV^{\frac{1}{q}}(u(0-,\cdot),(t_0,t))\leq \frac{C_{g}t}{t_0}\leq 2C_g.
			\end{equation*}
			Since $u\mapsto f_+^{-1}(g(u))$ is H\"older continuous with exponent $\frac{1}{p+1}$, we get
			\begin{equation*}
				\abs{u(0+,t_1)-u(0+,t_2)}\leq C_{f,g,\norm{u_0}_\f}\abs{u(0-,t_1)-u(0-,t_2)}^{\frac{1}{p+1}}.
			\end{equation*}
			Subsequently, we have
			\begin{equation*}
				TV^{s}(u(0+,\cdot),(t_0,t))\leq C_{f,g,\norm{u_0}_\f}\mbox{ where }s=\frac{1}{q(p+1)}.
			\end{equation*}
			Note that for $x\in(0,R_1(t))$ we have $u(x,t)=u(0+,t_+(x,t))$. Therefore, 
			\begin{equation}\label{1st-est}
				TV^{s}(u(\cdot,t),(0,R_1(t)))\leq C_{f,g,\norm{u_0}_\f}.
			\end{equation}
			For $x>R_1(t)$ we have $u(x,t)=(f^{\p})^{-1}\left(\frac{x-z_+(x,t)}{t}\right)$ for a non-decreasing $x\mapsto z_+(x,t)$. By using flux condition \eqref{fluxc} of $f$, we obtain 
			\begin{equation}
				TV^{\frac{1}{p}}(u(\cdot,t),(R_1(t),M))\leq \frac{C_{f,g,\norm{u_0}_\f}M}{t}.
			\end{equation} 
			Hence,
			\begin{align*}
				TV^s(u(\cdot,t),(0,M))&\leq TV^s(u(\cdot,t),(0,R_1(t)))+\norm{2u}_{L^\f(\R)}^{\frac{1}{s}}+TV^s(u(\cdot,t),(R_1(t),M))\\
				&\leq C_{f,g,\norm{u_0}_\f}+\norm{2u_0}_{L^\f(\R)}^{\frac{1}{s}}+TV^{\frac{1}{p}}(u(\cdot,t),(R_1(t),M))\\
				&\leq C_{f,g,\norm{u_0}_\f}+\norm{2u_0}_{L^\f(\R)}^{\frac{1}{s}}+\frac{C_{f,g,\norm{u_0}_\f}M}{t}.
			\end{align*}
			Next we consider the case when $t_0<t/2$. Let $x_0=\sup\{x;\,t_+(x,t)\geq t/2\}$. By Lemma \ref{lemma:partial-1} we have
			%	In this case, for $x>x_0$ 
			%	\begin{equation}
			%	u(x,t)=(f^\p)^{-1}\left(\frac{x-z_+(x,t;\si)}{t-\si}\right).%\mbox{ where }Y(x,t;\si)=y(x,t)+\si f^{\p}(u(x,t)).
			%	\end{equation}
			%	Note that $x\mapsto z_+(x,t;\si)$ is non-decreasing. Then by a similar argument we have
			\begin{equation}
				TV^{\frac{1}{p}}(u(\cdot,t);(x_0,M))\leq C_{f,g,\norm{u_0}_\f}+\frac{C_{f,g,\norm{u_0}_\f}M}{t}.
			\end{equation}
			Similar to \eqref{1st-est} we get
			\begin{equation*}
				TV^{s}(u(\cdot,t),(0,x_0))\leq C_{f,g,\norm{u_0}_\f}\mbox{ with }s=\frac{1}{q(p+1)}.
			\end{equation*}
			Subsequently, we obtain
			\begin{equation*}
				TV^s(u(\cdot,t),(0,M))\leq C_{f,g,\norm{u_0}_\f}+\norm{2u_0}_{L^\f(\R)}^{\frac{1}{s}}+\frac{C_{f,g,\norm{u_0}_\f}M}{t}.
			\end{equation*}
			Note that for $x<0$ we have $u(x,t)=(g^\p)^{-1}\left(\frac{x-z_-(x,t)}{t}\right)$. Then by using flux condition \eqref{fluxc} we can show that
			\begin{equation}
				TV^{\frac{1}{q}}(u(\cdot,t);(-M,0))\leq \frac{C_{f,g,\norm{u_0}_\f}M}{t}.
			\end{equation}
			The other case when $L_1(t)<0$ follows from a similar argument. This completes the proof of Theorem \ref{theorem:gen:reg}.
		\end{proof}
	}
	\subsection{Propagation of the initial regularity outside the interface}\label{2.4}
	The regularity of entropy solutions outside the interface is better than at the interface. It is proven in this section. 
	\begin{proof}[Proof of Theorem \ref{eps}]
		We consider the partition $\e\le x_{0}<x_{1}<\cdots<x_{l}\le R_{1}(t)\le x_{l+1}<\cdots$. Then
		\begin{eqnarray*}
			\sum_{i=0}^{\f}|u(x_{i},t)-u(x_{i+1},t)|^{1/s}&=&\sum_{i=0}^{l-1}|u(x_{i},t)-u(x_{i+1},t)|^{1/s}+\sum_{i=l}^{\f}|u(x_{i},t)-u(x_{i+1},t)|^{1/s}.
		\end{eqnarray*}
		Now from Theorem \ref{constant} we get,
		\begin{eqnarray*}
			\sum_{i=0}^{\f}|u(x_{i},t)-u(x_{i+1},t)|^{1/s}&\le&\sum_{i=0}^{l-1}\bigg|(f')^{-1}\bigg(\frac{x_{i}}{t-t_{+}(x_{i},t)}\bigg)-(f')^{-1}\bigg(\frac{x_{i+1}}{t-t_{+}(x_{i+1},t)}\bigg) \bigg|^{1/s}\\
			&+&\big|u(x_{l},t)-u(x_{l+1},t)\big|^{1/s}+\sum_{i=l+1}^{\f}\big|u_{0}(y(x_{i},t))-u_{0}(y(x_{i+1},t)) \big|^{1/s}.
		\end{eqnarray*}
		Since $t_{+}(x,t)$ is monotone function and has bound and infimum of $t-t_{+}(x, t)$ positive. Hence, we get
		\begin{equation*}
			\frac{\e}{T}\le\frac{x}{t-t_{+}(x, t)}\le\frac{M}{h(\e,T)},
		\end{equation*}
		where $h(\e,T)=\inf\{t-t_{+}(x,t): \e\le x\le R_{1}(t), 0<t\le T \}$, which also implies that  $(f')^{-1}$ is Lipschitz continuous function on interval $\left[\frac{\e}{T}, \frac{M}{h(\e,T)}\right]$. Then,
		\begin{eqnarray*}
			\sum_{i=0}^{\f}|u(x_{i},t)-u(x_{i+1},t)|^{1/s}&\le&C(\e,t)\sum_{i=0}^{l-1}\bigg|\frac{x_{i}}{t-t_{+}(x_{i},t)}-\frac{x_{i+1}}{t-t_{+}(x_{i+1},t)} \bigg|^{1/s}\\
			&+&\big|u(x_{l},t)-u(x_{l+1},t)\big|^{1/s}\\
			&+&\sum_{i=l+1}^{\f}\big|u_{0}(y(x_{i},t))-u_{0}(y(x_{i+1},t)) \big|^{1/s}.%,\\
			%&\le&C\bigg(\frac{R_{1}(t) -\e}{|t-t_{+}(\e,t)|}+\frac{R_{1}(t)|t_{+}(\e,t)-t_{+}(R_{1}(t),t)|}{|t-t_{+}(\e,t)|^{2}} \bigg)^{1/s}+  TV^{s}(u_{0})\\
			%&+&(2||u_{0}||)^{1/s},\\
			%&\le&C\sup_{0\le t\le T}\bigg(\frac{R_{1}(t) -\e}{|t-t_{+}(\e,t)|}+\frac{R_{1}(t)|t_{+}(\e,t)-t_{+}(R_{1}(t),t)|}{|t-t_{+}(\e,t)|^{2}} \bigg)^{1/s}\\
			%&+&  TV^{s}(u_{0})+(2||u_{0}||)^{1/s},\\
			%&\le& C(\e, t)+TV^{s}(u_{0})+(2||u_{0}||)^{1/s}.
		\end{eqnarray*} 
		%------------------------------------
		The estimate on first sum follow from,
		\begin{multline*}
			\sum_{i=0}^{l-1}\bigg|\frac{x_{i}}{t-t_{+}(x_{i},t)}-\frac{x_{i+1}}{t-t_{+}(x_{i+1},t)} \bigg|^{1/s}\\=\sum_{i=0}^{l-1}\left|\frac{x_{i}}{t-t_{+}(x_{i},t)}-\frac{x_{i}}{t-t_{+}(x_{i+1},t)}+\frac{x_{i}}{t-t_{+}(x_{i+1},t)}-\frac{x_{i+1}}{t-t_{+}(x_{i+1},t)} \right|^{1/s},\\
		\end{multline*}
		from triangle inequality we get,
		
		\begin{multline*}
			\sum_{i=0}^{l-1}\bigg|\frac{x_{i}}{t-t_{+}(x_{i},t)}-\frac{x_{i+1}}{t-t_{+}(x_{i+1},t)} \bigg|^{1/s}\\\le\sum_{i=0}^{l-1}\left(\left|\frac{x_{i}}{t-t_{+}(x_{i},t)}-\frac{x_{i}}{t-t_{+}(x_{i+1},t)}\right|+\left|\frac{x_{i}}{t-t_{+}(x_{i+1},t)}-\frac{x_{i+1}}{t-t_{+}(x_{i+1},t)} \right|\right)^{1/s} ,
		\end{multline*}
		now from the inequality $a^{1/s}+b^{1/s}\le(a+b)^{1/s}$ we get,
		\begin{multline*}
			\sum_{i=0}^{l-1}\bigg|\frac{x_{i}}{t-t_{+}(x_{i},t)}-\frac{x_{i+1}}{t-t_{+}(x_{i+1},t)} \bigg|^{1/s}\\\le\left(\sum_{i=0}^{l-1}\left|\frac{x_{i}}{t-t_{+}(x_{i},t)}-\frac{x_{i}}{t-t_{+}(x_{i+1},t)}\right|+\left|\frac{x_{i}}{t-t_{+}(x_{i+1},t)}-\frac{x_{i+1}}{t-t_{+}(x_{i+1},t)} \right|\right)^{1/s}.
		\end{multline*}
		Therefore, we get the following estimate,
		\begin{multline*}
			\sum_{i=0}^{l-1}\bigg|\frac{x_{i}}{t-t_{+}(x_{i},t)}-\frac{x_{i+1}}{t-t_{+}(x_{i+1},t)} \bigg|^{1/s}\le\left(\frac{R_{1}(t) -\e}{|t-t_{+}(\e,t)|}+\frac{R_{1}(t)|t_{+}(\e,t)-t_{+}(R_{1}(t),t)|}{|t-t_{+}(\e,t)|^{2}}\right)^{1/s}.
		\end{multline*}
		Thus we have,
		\begin{eqnarray*}
			\sum_{i=0}^{\f}|u(x_{i},t)-u(x_{i+1},t)|^{1/s}&\le&C\sup_{0\le t\le T}\bigg(\frac{R_{1}(t) -\e}{|t-t_{+}(\e,t)|}+\frac{R_{1}(t)|t_{+}(\e,t)-t_{+}(R_{1}(t),t)|}{|t-t_{+}(\e,t)|^{2}} \bigg)^{1/s}\\
			&+&  TV^{s}(u_{0})+(2||u_{0}||)^{1/s},\\
			&\le& C(\e, t)+TV^{s}(u_{0})+(2||u_{0}||)^{1/s}.
		\end{eqnarray*}
		In a similar way the other case $x\le-\e$ can be handle, 
		\begin{equation*}
			\sum_{i=0}^{\f}|u(x_{i},t)-u(x_{i+1},t)|^{1/s}\le C(\e, t)+TV^{s}(u_{0})+2(2||u_{0}||)^{1/s}.
		\end{equation*}
	\end{proof}
%	Here in the following section we are proving some results which we used in the proof our main theorems and results.
	%\section{Appendix}
	\section{Construction of counter-example}\label{sec:opt}
	%%%%%%%%%%%%%%%%%%%
	%%%%%%%%%%%%%%%%%%%%
	%Before proving the Theorem \ref{example}, we prove a technical lemma which we will need later in the proof of Theorem \ref{example}.
	Next we proceed to construct a counter-example so that initial data is in $BV$ but the  corresponding  solution is not in $BV^s$  at a fixed positive time $T>0$ and for some specific choice of flux. We refer to the backward construction for conservation laws with discontinuous flux introduced in \cite{AG}. In order to apply the method of backward construction we need to recall some notations and results from \cite{AG}. Next result is borrowed from \cite{AG} which says that given $h_+,z$ functions we can construct an entropy solution satisfying Hopf-Lax type formula for \eqref{1.1} with $h_+,z$.
	
	%Now we recall a result from \cite{AG} which guarantees the given $h_+,z$ functions we can construct an entropy solution.
	\begin{proposition}[Backward construction, \cite{AG}]\label{Prop:BC}
		Let $f,g$ are $C^1$ strictly convex functions. Let $R>0$ and $z:[0,R]\rr(-\f,0]$ be a non-decreasing function with $z_0=z(0+)$ and $z_1=z(R-)$. Suppose 
		\begin{align}
			h_+\left(\frac{R}{T-t_1}\right)&=-\frac{z_1}{t_1},\nonumber\\
			g^{\p}(u_-)&=\frac{z_0}{T},\,g^{\p}(v_-)=-\frac{z_1}{t_1},\,\bar{v}_-=f_+^{-1}(g(v_-)),
		\end{align}
		where $h_+$ is defined as
		\begin{equation}
			h_+:=g^{\p}\circ g_+^{-1}\circ f\circ (f^{\p})^{-1}.
		\end{equation}
		We additionally assume that $h_+$ is a locally Lipschitz function. Then there exists an initial data $u_0\in L^\f(\R)$ and the corresponding entropy solution $u$ to \eqref{1.1} such that
		\begin{equation}\label{Prop:cond}
			u(x,T)=(f^{\p})^{-1}\left(\frac{x}{T-t_+(x)}\right)\mbox{ where }-\frac{z(x)}{t_+(x)}=h_+\left(\frac{x}{T-t_+(x)}\right)\mbox{ for }x\in [0,R]
		\end{equation}
		and additionally, it holds $u(x,T)=u_-$ for $x<0$ and $u(x,T)=\bar{v}_-$ for $x>R$.
	\end{proposition}

	\begin{center}
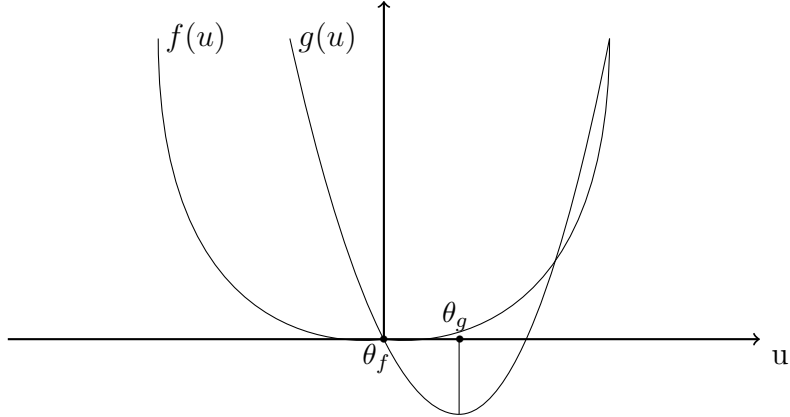
\label{figure-3}
		\begin{tikzpicture}[scale=1]
			\draw [thick, ->] (-5,0)--(5,0) node[anchor=north west] {u};
			\draw [thick, ->](0,0)--(0,4.5);%node[anchor=south east] {};
			\draw (0,0).. controls(0,0) and (3,-.5)..(3,4);
			\draw (0,0).. controls(0,0) and (-3,-.5)..(-3,4);
			%\draw (0,0) parabola (2,4);
			%\draw (0,0) parabola (-2, 4);
			\draw (1,-1) parabola (3,4);
			\draw (1,-1) parabola (-1.25,4);
			%numbering
			\draw(0,-.005) node{\tiny$\bullet$};
			\draw(1.01,0) node{\tiny$\bullet$};
			\draw(-.1,-.26) node{\small$\theta_{f}$};
			\draw (1,-1)--(1,0);
			\draw(.95,.3) node{\small$\theta_{g}$};
			\draw (-2.5,4) node{$f(u)$};
			\draw (-0.75,4) node{ $g(u)$};
		\end{tikzpicture}
		\captionof{figure}{An illustration of fluxes}
	\end{center}
	%\subsection{Proof of Proposition \ref{example}}
	To be self-contained the main ingredients of the proof is given in \ref{ap:opt}.\\
	Now the proof of  Theorem \ref{example} can be done.
	\begin{proof}[Proof of Theorem \ref{example}]
		Let $f(u)=\abs{u}^{p+1}$ and $g(u)=u^2-1$. Note that by Lemma \ref{new} $f$ satisfies the non-degeneracy condition  \eqref{fluxc} with exponent $p$ and $g$ is uniformly convex. 
		
		Let $\{a_k\}_{k\geq1}$ be a sequence defined as $a_{2i}=i^{-\B}$ and $a_{2i+1}=i^{-\al}$ with $\B>\al>0$ which will be chosen later. Consider an increasing sequence $\{t_k\}$ such that $t_k\rr 1$ and 
		\begin{equation}
			1-t_{2k+1}=\frac{1}{k^{\B-\al}}(1-t_{2k})\mbox{ and }t_{2k+2}-t_{2k+1}=k^{-\la}
		\end{equation}
		where $\la>1$ will be chosen later. Then we have
		\begin{equation}
			\frac{t_{2k+2}-t_{2k+1}}{t_{2k+1}}=\frac{1}{k^{\la}}\frac{1}{t_{2k+1}}\geq \frac{1}{k^{\la}}.
		\end{equation}
		We define $\{x_i\}$ as follows
		\begin{equation}\label{def:x_i}
			x_i=(1-t_{2i})a_{2i}=(1-t_{2i+1})a_{2i+1}.
		\end{equation}
		Since $\{t_{2i}\}_{i\geq1}$ is increasing and $\{a_{2i}\}_{i\geq1}$ is decreasing sequence, $\{x_i\}_{i\geq1}$ is a decreasing sequence. Let $h:[0,\f)\rr\R$ be defined as $$h(u)=2\sqrt{1+(p+1)^{-1-\frac{1}{p}}u^{1+\frac{1}{p}}}$$ for $u\geq 0$. Observe that
		\begin{equation}
			\frac{h(a_{2i+1})}{h(a_{2i+2})}-1=\frac{\sqrt{1+\left(\frac{i^{-\al}}{p+1}\right)^{1+\frac{1}{p}}}-\sqrt{1+\left(\frac{(i+1)^{-\B}}{p+1}\right)^{1+\frac{1}{p}}}}{\sqrt{1+\left(\frac{(i+1)^{-\B}}{p+1}\right)^{1+\frac{1}{p}}}}\leq \frac{1}{i^{\frac{p+1}{p}\al}}-\frac{1}{i^{\frac{p+1}{p}\B}}.
		\end{equation}
		Then if $\la<\frac{p+1}{p}\al$ we get
		\begin{equation}
			\frac{h(a_{2i+1})}{h(a_{2i+2})}-1<\frac{t_{2i+2}}{t_{2i+1}}-1.
		\end{equation}
		Therefore, we have 
		\begin{equation}\label{ineq:ht-1}
			t_{2i+1}h(a_{2i+1})<t_{2i+2}h(a_{2i+2}).
		\end{equation} 
		Note that
		\begin{align}
			\frac{1-t_{2i+1}}{1-t_{2i}}=\frac{1}{i^{\B-\al}}<1.
		\end{align}
		Hence,  $t_{2i+1}>t_{2i}$. Since $h(a_{2i+1})>h(a_{2i})$ we have $t_{2i+1}h(a_{2i+1})>h(a_{2i})t_{2i}$. Let $\xi(x)$ be solving the following problem
		\begin{align}
			\left(\frac{x}{1-\xi(x)}\right)^{1+\frac{1}{p}}&=\left(\frac{C}{\xi(x)+d}\right)^2-1\label{def:xi}\\
			\xi(x_{i})&=t_{2i+1},\label{condition:xi}\\
			\xi(x_{i+1})&=t_{2i+2}.\label{condition:xi1}
		\end{align} 
		Note that $C,d>0$ is determined by \eqref{condition:xi} and \eqref{condition:xi1}. Next we  show that $\xi^{\p}<0$. To this end we differentiate both side of \eqref{def:xi} and get the following
		\begin{equation}
			\begin{array}{rl}
				0<\left(1+\frac{1}{p}\right)x^{\frac{1}{p}}&=-\xi^{\p}(x)\left(1+\frac{1}{p}\right)(1-\xi(x))^{\frac{1}{p}}\left[\left(\frac{C}{\xi(x)+d}\right)^2-1\right]\\
				&-\xi^{\p}(x)(1-\xi(x))^{1+\frac{1}{p}}\frac{2C}{(\xi(x)+d)^2}\left(\frac{C}{\xi(x)+d}\right).
			\end{array}
		\end{equation}
		Therefore, we get $\xi^{\p}(x)<0$. Let $\Phi(x)$ be defined as 
		\begin{equation}
			\Phi(x):=\xi(x)\sqrt{1+\left(\frac{x}{1-\xi(x)}\right)^{1+\frac{1}{p}}}=\frac{C\xi(x)}{\xi(x)+d}.
		\end{equation}
		Observe that
		\begin{equation}\label{ineq-ht-2}
			\Phi^{\p}(x)=\xi^{\p}(x)\left[\frac{C}{\xi(x)+d}-\frac{C\xi(x)}{(\xi(x)+d)^2}\right]=\xi^{\p}(x)\frac{Cd}{(\xi(x)+d)^2}<0.
		\end{equation}
		Finally we define the function $t(x)$ such that $t(x_{i}+)=t_{2i}$ and $t(x_{i}-)=t_{2i+1}$ for $i\geq i_0$ and $t$ satisfies \eqref{def:xi}--\eqref{condition:xi1} for $x\in(x_{i+1},x_{i})$.  Let $\rho:(0,\f)\rr\R$ be defined as 
		\begin{equation}
			\rho(x)=-t(x)h\left(\frac{x}{1-t(x)}\right).
		\end{equation}
		By \eqref{ineq:ht-1} and \eqref{ineq-ht-2}, $x\mapsto \rho(x)$ is increasing. By Proposition \ref{Prop:BC} with $R=x_1$, there exists an entropy solution $u$ such that
		\begin{equation}
			u(x_{i}+,1)=\left(\frac{x_{i}}{(p+1)(1-t_{2i})}\right)^{\frac{1}{p}}\mbox{ and }u(x_{i}-,1)=\left(\frac{x_{i}}{(p+1)(1-t_{2i+1})}\right)^{\frac{1}{p}}.
		\end{equation}
		By \eqref{def:x_i} we get
		\begin{equation}
			u(x_{i}+,1)=\left(\frac{a_{2i}}{p+1}\right)^{\frac{1}{p}}\mbox{ and }u(x_{i}-,1)=\left(\frac{a_{2i+1}}{p+1}\right)^{\frac{1}{p}}.
		\end{equation}
		Therefore,
		\begin{align}
			\abs{u(x_{i}-,1)-u(x_i+,1)}&=\abs{\left(\frac{a_{2i}}{p+1}\right)^{\frac{1}{p}}-\left(\frac{a_{2i+1}}{p+1}\right)^{\frac{1}{p}}}\nonumber\\
			&=(1+p)^{-\frac{1}{p}}\left[i^{-\frac{\al}{p}}-i^{-\frac{\B}{p}}\right].%\frac{i^{-\frac{\B}{p}}}{p+1}\left[\left(1+\frac{1}{i^{{\al-\B}}}\right)^{\frac{1}{p}}-1\right]\geq \frac{C_1}{p+1}\frac{1}{i^{\al-\B+\frac{\B}{p}}}.
		\end{align}
		Let $\e>0$. Then, we have
		\begin{equation}
			\abs{u(x_{i}-,1)-u(x_i+,1)}^{\frac{p+1}{1+\e}}\geq C(p)\left[i^{-\frac{\al(p+1)}{p(1+\e)}}-i^{-\frac{\B(p+1)}{p(1+\e)}}\right].
		\end{equation}
		Now, we set
		\begin{equation}
			\la=1+\frac{2p}{3(2p+1)}\e,\,\frac{p+1}{p}\al=1+\frac{4p+2}{3(2p+1)}\e\mbox{ and }\frac{p+1}{p}\B=1+\frac{2(3p+2)}{3(2p+1)}\e.
		\end{equation}
		We check that $\beta-\al=\la-1$ and $\frac{p+1}{p}\B>1+\e$.	Hence, $u(\cdot,1)\notin BV^{s}_{loc}(\R)$ for $s=\frac{1}{p+1}+\frac{\e}{p+1}$. Note that by Proposition \ref{Prop:BC} initial data $u_0\in L^\f(\R)$. Now we find a data which is in $BV(\R)$. From the construction we have $x_1<R_2(1)$ where $R_2(t)$ is as in Theorem \ref{constant}. Choose a point $x_0\in (x_1,R_2(1))$. Note that $0<t_+(x_0,1)<1$ and $u(x,t_+(x_0,1))=\bar{v}_-$ for $x\geq0$. We also observe that $L_1(t)=0$ and $R_2(t)>0$ for $t=t_+(x_0,1)$. Therefore, for $t=t_+(x_0)$ we have
		\begin{equation}
			u(x,t)=(g^\p)^{-1}\left(\frac{x-z_-(x,t)}{t}\right)\mbox{ for }x<0.
		\end{equation}
		Since $g$ is uniformly convex we have $u(\cdot,t_+(x_0,1))\in BV((-\f,0))$. To conclude the Theorem \ref{example} we set $v_0(x):=u(x,t_0(x_0,1))$. Let $v(x,t)$ be the entropy solution to \eqref{1.1} with initial data $v_0$. Note that $v(x,1-t_0(x_0,1))=u(x,1)$ for all $x\in\R$. Hence, the proof of Theorem \ref{example} is completed.
	\end{proof}

	\appendix
	\section{H\"older continuity of singular maps}\label{appendix-a}
	
	In this section useful lemma on H\"older exponent and non degeneracy of fluxes are collected and used throughout the paper. Some commentaries are added for all lemmas. 
	
The following lemma recall that the non uniform convexity of  a flux  function corresponds to a loss of the Lipschitz regularity for the reciprocal function of the derivative.  This key point inforces a $BV^s$ (or generalized $BV$ regularity \cite{CJLO,GJC}) instead of $BV$ regularity \cite{lax1,ol} for the entropy solutions.
	\begin{customlemma}{A.1}\label{lemma:Holder}
		Let $g\in C^1(\R)$ be satisfying the non-degeneracy \eqref{fluxc} with exponent $q$. Then $(g^{\p})^{-1}$ is H\"older continuous with exponent $1/q$.
	\end{customlemma}
	\begin{proof}
		Fix a compact set $K$. Let $x$ and $y$ is in $g'(K)$. There exist $\tilde{x},\tilde{y}$ such that $\tilde{x}=(g')^{-1}(x)$ and $\tilde{y}=(g')^{-1}(y)$. Then, %.  For proving this fact, assume that $x$ and $y$ is in $g'(K)$ for any compact set $K$, with $\tilde{x}=(g')^{-1}(x)$ and $\tilde{y}=(g')^{-1}(y)$, then,
		\begin{equation*}
			\frac{|(g')^{-1}(x)-(g')^{-1}(y)|}{|x-y|^{1/q}}=\frac{|\tilde{x}-\tilde{y}|}{|g'(\tilde{x})-g'(\tilde{y})|^{1/q}}=\frac{|\tilde{x}-\tilde{y}|}{|g'(\tilde{x})-g'(\tilde{y})|^{1/q}}\le\frac{1}{C_{2}^{1/q}}.
		\end{equation*} 
		This proves the Lemma \ref{lemma:Holder}.
	\end{proof}
	
		The interface condition \eqref{RH} needs the use of some reciprocal functions of the flux $g$ or $f$. The fact that the reciprocal function of $g$  is never Lipschitz near $\min g$ forbids the classical  Lax-Oleinik $BV$ smoothing effect for an uniform convex flux.  
	\begin{customlemma}{A.2}\label{lemma:g-condition}
		Let $g$ be a $C^2$ function satisfying \eqref{fluxc} with exponent $q$ then $g_+$ satisfies \eqref{fluxc} with exponent $q+1$ on domain $(\theta_g,\f)$.
	\end{customlemma}
	\begin{proof}
		Since $\theta_{g}$ is the critical point of $g$ hence, $g'(\theta_{g})=0$, then we consider 
		\begin{align*}
			g(x)-g(y)&=(x-y)\int_{0}^{1}g'(\lambda x+(1-\lambda)y)d\lambda,\\
			&=(x-y)\int_{0}^{1}(g'(\lambda x+(1-\lambda)y)-g'(\theta_{g}))d\lambda.
		\end{align*}
		We know that $g'(\cdot)$ is a increasing function and $g$ satisfies  the non-degeneracy condition \eqref{fluxc}. Let $x>y\ge \theta_{g}$, then
		\begin{align}\nonumber
			|g(x)-g(y)|&=|x-y|\int_{0}^{1}(g'(\lambda x+(1-\lambda)y)-g'(\theta_{g}))d\lambda\\\nonumber
			&\ge C_{2}|x-y|\int_{0}^{1}(\lambda x+(1-\lambda)y)-\theta_{g})^{q}d\lambda\\\nonumber
			&\ge \frac{1}{q+1}C_{2}\big((x+(1-\lambda)y)-\theta_{g})^{q+1}\big)\Bigg|_{0}^{1} \\\nonumber
			&\ge\frac{1}{q+1}C_{2}((x-\theta_{g})^{q+1}-(y-\theta_{g})^{q+1})\\
			&\ge\frac{C_{2}}{q+1}|x-y|^{q+1}.\label{q+1}
		\end{align} 
	\end{proof}
	
		The previous commentary of Lemma \ref{lemma:g-condition}   is even more important for the non Lipschitz regularity of the singular map.
	\begin{customlemma}{A.3}\label{lipschitz}
		Suppose fluxes $f$ and $g$ are $C^{1}(\R)$ and convex functions with $f(\theta_{f})<g(\theta_{g})$ which additionally satisfies the  non-degeneracy condition \eqref{fluxc} and let $K$ is any compact set of $\R$. Then for $x\in K$, $f^{-1}_{+}g(\cdot)$ is a Lipschitz continuous function and $g^{-1}_{-}f(\cdot)$ is a H\"older continuous function.
	\end{customlemma}
	\begin{proof}
		%{\color{red}We know that for given $g(x)$ there exists $\tilde{x}\in\R$ away from $\theta_{f}$ such that $g(x)=f(\tilde{x})$ and $f'(\tilde{x})>0$.}
		Since $f(\theta_{f})<g(\theta_{g})$, there exist $a_1<\theta_f<a_2$ such that $f(a_1)=g(\theta_g)=f(a_2)$. Hence, we have
		\begin{equation}
			\bar{c}:=\min\left\{\abs{f^{\p}(a)};\,a\in(-\f,a_1]\cup[a_2,\f)\right\}>0.
		\end{equation} Without loss generality we can assume that $g(x)\not=g(y)$ because if $g(x)=g(y)$ then result holds anyway. There exist $\tilde{x},\tilde{y}>\theta_f$ such that $f(\tilde{x})=g(x)$ and $f(\tilde{y})=g(y)$. As $f^{-1}_+$ is increasing, we get $\tilde{x},\tilde{y}>a_2$. Consider the following 
		\begin{eqnarray*}
			\frac{|f^{-1}_{+}g(x)-f^{-1}_{+}g(y)|}{|x-y|}&=&\frac{|f^{-1}_{+}g(x)-f^{-1}_{+}g(y)|}{|g(x)-g(y)|}\cdot\frac{|g(x)-g(y)|}{|x-y|},\\
			&=&\frac{|f^{-1}_{+}f(\tilde{x})-f^{-1}_{+}f(\tilde{y})|}{|f(\tilde{x})-f(\tilde{y})|}\cdot\frac{|g(x)-g(y)|}{|x-y|},\\
			&=&\frac{|\tilde{x}-\tilde{y}|}{|f(\tilde{x})-f(\tilde{y})|}\cdot\frac{|g(x)-g(y)|}{|x-y|},\\
			&=&\frac{1}{f'(c_{0})}\cdot\frac{|g(x)-g(y)|}{|x-y|},
		\end{eqnarray*} 
		%due to the fact that the slope of line joining point $(\tilde{x}, f(\tilde{x}))$ and $(\tilde{y}, f(\tilde{y}))$ is bounded because
		for some $c_0$ in between $\tilde{x},\tilde{y}$. Note that ${c}_0\geq a_2$ and $f^{\p}\geq \bar{c}$. As $g$ is Lipschitz continuous function, we have $\abs{g(x)-g(y)}\leq c_1\abs{x-y}$ where $c_1$ depends on $g$ and $K$. % since $f'$ is increasing function and $f(\tilde{x}), f(\tilde{y})\ge g(\theta_{g})>f(\theta_{f})$ which implies that $f'$ is bounded below.
		Therefore we get,
		\begin{equation}
			\frac{|f^{-1}_{+}g(x)-f^{-1}_{+}g(y)|}{|x-y|}\le C.
		\end{equation}
		We know that for $f(x)\ge g(\theta_{g})$ there exists $\tilde{x}$ such that $f(x)=g(\tilde{x})$ and $g'(\tilde{x})>0$, without loss of generality we can assume that $g(x)\not=g(y)$ because if $g(x)=g(y)$ then result holds.
		\begin{eqnarray*}
			\frac{|g^{-1}_{-}f(x)-g^{-1}_{-}f(y)|^{q+1}}{|x-y|}&=&\frac{|g^{-1}_{-}f(x)-g^{-1}_{-}f(y)|^{q+1}}{|f(x)-f(y)|}\cdot\frac{|f(x)-f(y)|}{|x-y|},\\
			&=&\frac{|g^{-1}_{-}g(\tilde{x})-g^{-1}_{-}g(\tilde{y})|^{q+1}}{|g(\tilde{x})-g(\tilde{y})|}\cdot\frac{|f(x)-f(y)|}{|x-y|}
			,\\
			&=&\frac{|\tilde{x}-\tilde{y}|^{q+1}}{|g(\tilde{x})-g(\tilde{y})|}\cdot\frac{|f(x)-f(y)|}{|x-y|},
		\end{eqnarray*}
		Now from the Lipschitz continuity $f$ and \eqref{q+1},
		\begin{equation}
			\frac{|g^{-1}_{-}f(x)-g^{-1}_{-}f(y)|^{q+1}}{|x-y|}\le C.
		\end{equation}
		Hence, it implies that
		\begin{equation*}
			|g^{-1}_{-}f(x)-g^{-1}_{-}f(y)|\le C|x-y|^{1/q+1}.
		\end{equation*}
	\end{proof}
Next lemma shows that power law fluxes satisfies the non-degeneracy condition \eqref{fluxc}. 
	\begin{customlemma}{A.4}\label{new}
	Let $M>0$ and $g:[-M,M]\rr\R$ be defined as $g(u)=\abs{u}^p$ for $p\geq2$. Then $g$ satisfies the non-degeneracy condition \eqref{fluxc} with exponent $p-1$.
\end{customlemma}
This is the simplest example with power-law degeneracy $p-1$ \cite{junca1,CJ1}.
\begin{proof}We calculate $g'(u)=\sign(u)p|u|^{p-1}$. Then we show  that the non-degeneracy condition \eqref{fluxc}  is satisfied for $g$ case by case. Hence, we consider \\
	Case(\RN{1}): If $u,v\ge0$, since  $(|u|+|v|)^{p}>(|u|^{p}+|v|^{p})$, we get
	\begin{equation}
		\frac{|g'(u)-g'(v)|}{|u-v|^{p-1}}=p\frac{|u^{p-1}-v^{p-1}|}{|u-v|^{p-1}}\ge p.
	\end{equation}
	Now for $u,v\le 0$ can be handled similarly as Case \RN{1}.\\
	Case(\RN{2}): If $u\le0$ and $v\ge0$, then we  also get
	\begin{equation}
		\frac{|g'(u)-g'(v)|}{|u-v|^{p-1}}=p\frac{||u|^{p-1}+|v|^{p-1}|}{|u-v|^{p-1}}\ge p.
	\end{equation}
	Again for $u\le0$ and $v\ge0$ can be handled in similar way as Case \RN{2}. 
\end{proof}
	\section{$BV^{s}$ embedding}
The continuous embedding between fractional $BV$ spaces is explicited using the $L^\infty$ norm  or, more precisely,  the oscillation in the next lemma. 
Recall that   the oscillation of the function $u$ on $I$ is,
 $$ osc(u):=\sup_{x<y}\{|u(x) -u(y)|\} \leq  2 \|u \|_\infty .$$ 
\begin{customlemma}{B.1}\label{TVs_norm}
			Let $u:I\subset\R\rr\R$ be bounded function on a given interval $I$ and $0<s<t$ such that $u\in BV^{t}\subset BV^{s}$. Let $p=\frac{1}{s} \geq q=\frac{1}{t}$, then,
			\begin{equation}
				TV^{s}u(I)  \; \le  \; osc(u)^{p-q} \; TV^{t}u(I) .
				% \qquad   D =  \leq   (2 \|u\|_\infty) ^{p-q} .
			\end{equation} 
		\end{customlemma}
		\begin{proof}
 When $osc(u) \leq 1$,  the inequality  $ y^p \leq y^q$ for all $y \in [0,1]$ gives a direct estimate.  More precisely,
		let $\sigma=(x_{1},\cdots,x_{n})$ be any partition of $I$,
			\begin{align*}
				\sum\limits_{i=1}^{n-1}|u(x_{i})-u(x_{i+1})|^{p}\le \sum\limits_{i=1}^{n-1}|u(x_{i})-u(x_{i+1})|^{q}
		\le  TV^{t}u(I).
			\end{align*}
This inequality can be improved as follows if $u$ is non constant, that is $osc(u) > 0$. 
For this purpose, consider   $  v=u /osc(u)$  so $ osc(v) \leq 1 $.  Now, on a subdivision, we have, 
			\begin{align*}
	osc(u)^{-p}	\sum\limits_{i=1}^{n-1}|u(x_{i})-u(x_{i+1})|^{p}
		=
		\sum\limits_{i=1}^{n-1}|v(x_{i})-v(x_{i+1})|^{p}  
		\\  \le \sum\limits_{i=1}^{n-1}|v(x_{i})-v(x_{i+1})|^{q}
		=  osc(u)^{-q}	\sum\limits_{i=1}^{n-1}|u(x_{i})-u(x_{i+1})|^{q}.
			\end{align*}
			That is to say,  the following inequality which is also valid when $osc(u)=0$,
			\begin{align*}
\sum\limits_{i=1}^{n-1}|u(x_{i})-u(x_{i+1})|^{p}	
	 \le
		  osc(u)^{p-q}	\sum\limits_{i=1}^{n-1}|u(x_{i})-u(x_{i+1})|^{q}. 
			\end{align*}
			This is enough to conclude the lemma. 
	\end{proof}
	%%%%%%%%%%%%%%%%%%%%%%%%%%%%%%%
	\section{Backward construction} \label{ap:opt}
	%%%%%%%%%%%%%%%%%%%%%%%%%%%%
	The proof of the optimality presented in section \ref{sec:opt} needs a construction of an initial data and solution by borrowing ideas and techniques from  control. We  only give a sketch of the  existence of such solution along with initial data that stated in Proposition \ref{Prop:BC}. The complete construction can be found in \cite{AG}.

	\begin{proof}[Proof of Theorem \ref{Prop:BC}]
		We first approximate $z(x)$ by piece-wise constant increasing function as follows
		\begin{equation}
			\left\{\begin{array}{r} % {rl}
				z_0 = w_0 < w_1 < \cdots < w_k = z_1,\\
				\abs{w_{i+1}-w_{i}} <\frac{1}{N},\\
				0 = x_0 < x_1 <\cdots< x_k = R,\\
				z(x_i) = w_i\mbox{ for }1\leq i\leq k-1,\\
				\mbox{ with } z_0 = z(0) \mbox{ and } z_1 = z(R-). 
			\end{array}\right.
		\end{equation}
		We set $t_0=T$ and $t_i,\,1\leq i\leq 2k,$, $c_i,d_i\,1\leq i\leq k$ as follows
		\begin{equation}
			\begin{array}{rl}
				h_+\left(\frac{x_i}{T-t_{2i-1}}\right)=-\frac{w_{i-1}}{t_{2i-1}},\,h_+\left(\frac{x_i}{T-t_{2i}}\right)=-\frac{w_{i}}{t_{2i}},\\
				f^{\p}(c_{2i-1})=\frac{x_{i}}{T-t_{2i-1}},\,	f^{\p}(c_{2i})=\frac{x_{i}}{T-t_{2i}}\mbox{ and }d_i=g_+^{-1}(f(a_i)).
			\end{array}
		\end{equation}
		Then we observe that $c_{2i-1}>c_{2i},d_{2i-1}>d_{2i},T=t_0>t_1>\cdots>t_{2k}=T_1$. Consider Lipschitz curves $r_i,\tilde{r}_i,a_i,b_i$ defined as follows
		\begin{align}
			&\begin{array}{rll}
				s_i=\frac{f(c_{2i-1})-f(c_{2i})}{c_{2i-1}-c_{2i}},&S_i=\frac{g(d_{2i-1})-g(d_{2i})}{d_{2i-1}-d_{2i}},&1\leq i\leq k,\\
				r_i(t)=g^{\p}(d_i)(t-t_i),&\tilde{r}_i(t)=f^{\p}(c_i)(t-t_i),&1\leq i\leq 2k,\\
				a_i(t)=x_i+s_i(t-T),&b_i(t)=S_i(t-q_i),\,a_i(q_i)=0,&1\leq i\leq 2k,
			\end{array}\\
			&r_0(t)=g^{\p}(b_0)(t-T)=g^{\p}(u_-)(t-t_0).
		\end{align}
		Now, we define $u_{0}^N$ as below
		\begin{equation}
			u_0^N:=\left\{\begin{array}{ll}
				u_-&\mbox{ if }x<w_0,\\
				d_{2i-1}&\mbox{ if }w_{i-1}<x<b_i(0),1\leq i\leq k,\\
				d_{2i}&\mbox{ if }b_i(0)<x<w_{i},1\leq i\leq k,\\
				v_-&\mbox{ if }w_{2k}<x<0,\\
				\bar{v}_-&\mbox{ if }x>0.
			\end{array}\right.
		\end{equation}
		Let $\tilde{t}_{i}(x)$ be the unique solution to
		\begin{equation}
			h_+\left(\frac{x}{T-\tilde{t}_i(x,t)}\right)=-\frac{z_i}{\tilde{t}_i(x,t)}\mbox{ for }x\in(x_i,x_{i+1}),\,1\leq i\leq k-1.
		\end{equation} Corresponding entropy solution $u^N$ is the following
		\begin{equation}
			u^N(x,t)=\left\{\begin{array}{ll}
				u_-&\mbox{ if }x<r_0(t),\\
				(g^{\p})^{-1}\left(\frac{x-z_{i}}{t}\right)&\mbox{ if }r_{2i}(t)<x<\min\{r_{2i+1}(t),0\},\\
				(f^{\p})^{-1}\left(\frac{x}{t-\tilde{t}_i(x,t)}\right)&\mbox{ if }\max\{\tilde{r}_{2i+1}(t),0\}<x<\tilde{r}_{2i-1}(t),\\
				d_{2i-1}&\mbox{ if }r_{2i-1}(t)<x<\min\{S_i(t),0\},1\leq i\leq k,\\
				d_{2i}&\mbox{ if }S_{2i}(t)<x<\min\{r_{2i}(t),0\},1\leq i\leq k,\\
				c_{2i-1}&\mbox{ if }\max\{\tilde{r}_{2i-1}(t),0\}<x<s_i(t),1\leq i\leq k,\\
				c_{2i}&\mbox{ if }\max\{s_i(t),0\}<x<\tilde{r}_{2i},1\leq i\leq k,\\
				v_-&\mbox{ if }r_{2k}(t)<x<0,\\
				\bar{v}_-&\mbox{ if }x>\max\{\tilde{r}_{2k},0\}.
			\end{array}\right.
		\end{equation}
		By assumption we have $h_+$ is a locally Lipschitz continuous function and we can prove TV bound of $g^{\p}(u_0^N)$ (see \cite{AG} for more details). Then, by applying Helly's Theorem we can find a $u_0\in L^\f(\R)$ and corresponding entropy solution $u$ satisfying \eqref{Prop:cond}. This completes the proof of Proposition \ref{Prop:BC}.
	\end{proof}

	\noindent\textbf{Acknowledgement.} Authors thank IFCAM project ``Conservation laws: $BV^s$, interface and control".  SSG and AP thank the Department of Atomic Energy, Government of India, under project no. 12-R\&D-TFR-5.01-0520 for support. SSG acknowledges Inspire faculty-research grant DST/INSPIRE/04/2016/00-0237.


\begin{thebibliography}{99}
		
		
		\bibitem{AS}
		\newblock  Adimurthi, R. Dutta, S. S. Ghoshal and G. D. Veerappa Gowda, 
		\newblock Existence and nonexistence of TV bounds for scalar conservation laws with discontinuous flux, 
		\newblock {\it Comm. Pure Appl. Math.}, 64 (2011), no. 1, 84-115.
		
		\bibitem{finer}
		\newblock Adimurthi, S. S. Ghoshal and  G. D. Veerappa Gowda,
		\newblock Finer regularity of an entropy solution for 1-d scalar conservation laws with non uniform convex flux,
		\newblock{\it  Rend. Semin. Mat. Univ. Padova}, 132 (2014), 1-24.
		
		\bibitem{AG}
		\newblock Adimurthi and S. S. Ghoshal,% and  G. D. Veerappa Gowda,
		\newblock Exact and optimal controllability for scalar conservation laws with discontinuous flux,
		\newblock to appear in {\it Commun. Contemp. Math.}, (arxiv preprint arXiv:2009.13324).
		
		
		
		%	\bibitem{structure}
		%	\newblock Adimurthi, S. S. Ghoshal and  G. D. Veerappa Gowda,
		%	\newblock Structure of entropy solutions to scalar conservation laws with strictly convex flux,
		%	\newblock {\it  J. Hyperbolic Differ. Equ.}, 9 (2012), no. 4, 571-611.
		
		\bibitem{AJG}
		\newblock Adimurthi, J. Jaffr\'e and G. D. Veerappa Gowda,
		\newblock Godunov type methods for scalar conservation laws with flux function discontinuous in the space variable,
		\newblock {\it SIAM J. Numer. Anal.}, 42(1) (2004), 179-208. 
		
		
		
		
		\bibitem{ASG}
		\newblock Adimurthi, S. Mishra and G. D. Veerappa  Gowda,
		\newblock Optimal entropy solutions for conservation laws with discontinuous flux-functions,
		\newblock {\it J. Hyperbolic Differ. Equ.}, 2 (2005), no. 4, 783-837.
		
		\bibitem{Explicit}
		\newblock Adimurthi, S. Mishra and G. D. Veerappa  Gowda,
		\newblock Explicit Hopf-Lax type formulas for Hamilton-Jacobi equations and conservation laws with discontinuous coefficients,
		\newblock {\it  J. Differential Equations}, 241 (2007), no. 1, 1-31.
		
		\bibitem{Kyoto}
		\newblock Adimurthi and G. D. Veerappa Gowda,
		\newblock Conservation laws with discontinuous flux,
		\newblock {\it J. Math. Kyoto Univ.}, 43 (2003), no. 1,  27-70.
		
		
		\bibitem{AFP} L. Ambrosio, N. Fusco and D. Pallara, 
		\newblock Functions of bounded variation and free discontinuity problems,
		% (English) Zbl 0957.49001
		\newblock{\it Oxford Mathematical Monographs}, %Oxford: Clarendon Press.
		xviii, 434 p. (2000).
		
		
		
		\bibitem{boris1}
		\newblock  B. Andreianov and C. Canc\`es,
		\newblock The Godunov scheme for scalar conservation laws with discontinuous bell-shaped flux functions, 
		\newblock {\it  Appl. Math. Lett.}, 25 (2012), no. 11, 1844-1848.
		
		\bibitem{boris2}
		\newblock  B. Andreianov,  K. H. Karlsen and  N. H. Risebro,
		\newblock  A theory of $L^1$-dissipative solvers for scalar conservation laws with discontinuous flux,
		\newblock {\it Arch. Ration. Mech. Anal.} 201 (2011), no. 1, 27-86.
		
		
		%\bibitem{BM17}  S. Bianchini and E. Marconi, 
		%\newblock On the structure of $L^\infty$  entropy solutions to scalar conservation laws in one-space dimension,
		%\newblock {\it Ar. Mech. Anal.} 226 (2017), no.1, 441-493.
		%	
		
		
		\bibitem{junca1}
		\newblock C. Bourdarias, M. Gisclon and  S. Junca,
		\newblock Fractional BV spaces and applications to scalar conservation laws,
		\newblock {\it J. Hyperbolic Differ. Equ.} 11 (2014), no. 4, 655-677.
		
		%	\bibitem{junca2}
		%	\newblock  C. Bourdarias, M. Gisclon,  S. Junca and Y. J. Peng,
		%	\newblock Eulerian and Lagrangian formulations in $BV^{s}$ for gas-solid chromatography,
		%	\newblock {\it Comm. Math. Sci.} 14 (2016), no. 6, 1665-1685.
		
		
		
		\bibitem{BGG} A. Bressan, G. Guerra and W. Shen, 
		\newblock Vanishing viscosity solutions for conservation laws with regulated flux. 
		\newblock {\it J. Differ. Equ.} 266 (2019) 312–351.
		
		\bibitem{burger1}
		\newblock  R. B\"urger, A. Garc\'ia, K. H. Karlsen and  J. D. Towers,
		\newblock A family of numerical schemes for kinematic flows with discontinuous flux,
		\newblock{\it  J. Engrg. Math.}, 60 (2008), no. 3-4, 387-425.
		
		%	\bibitem{burger3}
		%	\newblock R. B\"urger, K. H. Karlsen, C. Klingenberg and N. H. Risebro,
		%	\newblock A front tracking approach to a model of continuous sedimentation in ideal clarifier-thickener units,
		%	\newblock{\it  Nonlinear Anal. Real World Appl.}, 4 (2003), no. 3, 457-481.
		
		\bibitem{burger2}
		\newblock R. B\"urger, K. H. Karlsen, N. H. Risebro and  J. D. Towers,
		\newblock  Well-posedness in $BV_{t}$ and convergence of a difference scheme for continuous sedimentation in ideal clarifier-thickener units,
		\newblock {\it  Numer. Math.}, 97 (2004), no. 1, 25-65.
		
		%	\bibitem{burger5}
		%	\newblock R. B\"urger, K. H. Karlsen and  N. H. Risebro,
		%	\newblock A relaxation scheme for continuous sedimentation in ideal clarifier-thickener units,
		%	\newblock{\it  Comput. Math. Appl.}, 50 (2005), no. 7, 993-1009.
		%	
		%	\bibitem{burger6}
		%	\newblock R. B\"urger, K. H. Karlsen, N. H. Risebro and  J. D. Towers,
		%	\newblock Monotone difference approximations for the simulation of clarifier-thickener units,
		%	\newblock {\it  Comput. Vis. Sci.}, 6 (2004), no. 2-3, 83–91.
		%	
		%	\bibitem{water}
		%	\newblock R. B\"urger, K. H. Karlsen, N. H. Risebro and J. D. Towers, 
		%	\newblock Well-posedness in $BV_{t}$ and convergence of a difference scheme for continuous sedimentation in ideal clarifier thickener units,
		%	\newblock {\it Numer. Math.,} 97 (2004), no. 1, 25-65. 
		%	
		\bibitem{burger4}
		\newblock  R. B\"urger, K. H. Karlsen and  J. D. Towers,
		\newblock A model of continuous sedimentation of flocculated suspensions in clarifier-thickener units,
		\newblock{\it  SIAM J. Appl. Math.}, 65 (2005), no. 3, 882-940.
		
		%	\bibitem{buger7}
		%	\newblock R. Bürger, K. H. Karlsen and J. D. Towers,
		%	\newblock  An Enquist-Osher-type scheme for conservation laws with discontinuous flux adapted to flux connections,
		%	\newblock {\it SIAM J. Numer. Anal.}, 47 (2009), no. 3, 1684-1712.
		
		
		\bibitem{CJ1} P.~Castelli and S.~Junca, 
		\newblock { Oscillating waves and the maximal smoothing effect for one dimensional nonlinear conservation laws}, 
		\newblock {\it AIMS on Applied Mathematics}, 8, 709-716, (2014).
		
		\bibitem{CJLO} P.~Castelli and  S.~Junca, 
		\newblock Smoothing effect in $BV-\Phi$ for entropy solutions of scalar conservation laws,
		\newblock % (hal.archives-ouvertes.fr/hal-01133725), 
		{\it  J. Math. Anal. Appl.}, 451 (2), 712–735,  (2017). 
		
		\bibitem{CJJ} P.~Castelli, P. E. Jabin and S.~Junca,
		\newblock Fractional spaces and conservation laws, 
		\newblock {\it Theory, numerics and applications of hyperbolic problems I, Aachen, Germany, August 2016.  Springer Proceedings in Mathematics \& Statistics}, 236, 285-293 (2018).
		
		
		
		\bibitem{Cheng83}  K. S. Cheng, 
		\newblock The space $BV$ is not enough for hyperbolic conservation laws,
		\newblock    {\it  J. Math. Anal. Appl.}, 91 (2), 559–561,  (1983). 
		
		%\bibitem{Cheng} K. S. Cheng,  
		%\newblock  A regularity theorem for a non convex scalar conservation law,
		%\newblock {\it  J. Differential Equations}, 61  (1), 79–-127,  (1986).
		%	
		%
		%\bibitem{DOW} C. De Lellis, F. Otto and M. Westdickenberg,  
		%\newblock {Structure of entropy solutions for multidimensional scalar conservation laws},
		%\newblock {\it Arch. Ration. Mech. Anal.},  170 (2), 137-184, (2003).
		%
		%\bibitem{DR}  C. De Lellis and T. Rivi\`ere,        
		%\newblock {The rectifiability of entropy measures in one space dimension}, 
		%\newblock {\it J. Math. Pures Appl.}, (9) 82, no. 10, 1343-1367, (2003).
		
		
		\bibitem{diehl}
		\newblock S. Diehl,
		\newblock  Dynamic and steady-state behavior of continuous sedimentation,
		\newblock{\it SIAM J. Appl. Math.}, 57 (1997), no. 4, 991-1018.
		
		%	\bibitem{diehl1}
		%	\newblock S. Diehl,
		%	\newblock Operating charts for continuous sedimentation. II. Step responses,
		%	\newblock {\it J. Engrg. Math.}, 53 (2005), no. 2, 139–185.
		
		%	\bibitem{diehl2}
		%	\newblock S. Diehl,
		%	\newblock A uniqueness condition for nonlinear convection-diffusion equations with discontinuous coefficients,
		%	\newblock {\it J. Hyper. Diff. Equ. }, 6 (2009), no. 1, 127-159.
		
		%	\bibitem{diehl3}
		%	\newblock S. Diehl,
		%	\newblock  Conservation laws with application to continuous sedimentation,
		%	\newblock{\it Doctoral Thesis Lund University (Sweden)}, 1995.
		
		\bibitem{diehl4}
		\newblock S. Diehl,
		\newblock A conservation law with point source and discontinuous flux function modeling continuous sedimentation,
		\newblock{\it  SIAM J. Appl. Math.}, 56 (1996), no. 2, 388–419.
		
		
		
		\bibitem{S}
		\newblock S. S. Ghoshal,
		\newblock Optimal results on TV bounds for scalar conservation laws with discontinuous flux,
		\newblock {\it J. Differential Equations}, 258 (2015), no. 3, 980-1014.
		
		\bibitem{S2}
		\newblock S. S. Ghoshal, 
		\newblock BV regularity near the interface for nonuniform convex discontinuous flux,
		\newblock {\it Netw. Heterog. Media}, 11 (2016), no. 2, 331-348.
		
		\bibitem{SAJJ}
		\newblock S. S.  Ghoshal, B.  Guelmame, A.  Jana and  S. Junca,
		\newblock Optimal regularity for all time for entropy solutions of conservation laws in $BV^{s}$,
		\newblock {\it Nonlinear Differential Equations and Applications NoDEA}, 27 (2020),  article number 46, 29 p.
		
		\bibitem{SA}
		\newblock  S. S. Ghoshal and  A.  Jana, 
		\newblock  Non existence of the BV regularizing effect for scalar conservation laws in several space dimension for $C^2$ fluxes, 
		\newblock {\it  SIAM J. Math. Anal.} 53 (2021), no. 2, 1908--1943.
		
		\bibitem{SAT}
		\newblock S. S. Ghoshal,  A. Jana and J. D Towers, 
		\newblock Convergence of a Godunov scheme to an Audusse-Perthame  adapted entropy solution for conservation laws with BV spatial flux,
		\newblock {\it Numer. Math.} 146 (3), (2020), 629-659.
		
		%	\bibitem{gimse}
		%	\newblock T. Gimse and N. H. Risebro,
		%	\newblock Solution of Cauchy problem for conservation law with discontinuous flux function,
		%	\newblock {\it SIAM J. Math. Anal.}, 23(3) (1992) 635-648.
		%	
		\bibitem{GJC} B. Guelmame, S. Junca and D. Clamond, 
		\newblock Regularizing effect for conservation laws with a Lipschitz convex flux,
		\newblock   {\it Commun. Math. Sci.}, 17 (8), 2223-2238, (2019).
		
		%	\bibitem{Hoff} D. Hoff, 
		%  \newblock  {The sharp form of Oleinik's entropy condition in several space variables},
		%\newblock {\it  Trans. Am. Math. Soc.}, 276, 707-714,  (1983). 
		
		\bibitem{JaX} P. E.~Jabin, 
		\newblock  Some regularizing methods for transport equations and the regularity of solutions to scalar conservation laws, 
		\newblock {\it S\'eminaire: Equations aux D\'eriv\'ees Partielles, Ecole Polytech. Palaiseau},
		2008-2009, Exp. No. XVI, (2010).
		
		
		
		\bibitem{petro}
		\newblock J. Jaffr\'e and S. Mishra,
		\newblock On the upstream mobility flux scheme for the simulating two phase flow in heterogeneous porous media,
		\newblock {\it Comput. Geosci.},  2009.
		%	\bibitem{god}
		%	\newblock E. Godlewski, P.-A. Raviart,
		%	\newblock Numerical approximation of hyperbolic systems of conservation laws. Applied Mathematical Sciences, 118.
		%	\newblock{\it Springer-Verlag, New York}, 1996. viii +509 pp.
		
		\bibitem{godunov}
		\newblock S. K. Godunov,
		\newblock A difference method for numerical calculation of discontinuous solutions of the equations of hydrodynamics. (Russian)
		\newblock{\it Mat. Sb. (N.S.)} 47 (1959) no. 89, 271-306.
		
		
		\bibitem{KT}
		\newblock  K. H. Karlsen and J. D. Towers,
		\newblock Convergence of a Godunov scheme for conservation laws with a discontinuous flux lacking the crossing condition.
		\newblock{\it  J. Hyperbolic Differ. Equ.}, 14 (2017), no. 4, 671--701.
		
		\bibitem{Kruzkov}
		\newblock S. N. Kru\v{z}kov,
		\newblock First-order quasilinear equations with several space variables,
		\newblock \textit{Mat. Sbornik,} 123 (1970), 228-255; Math. USSR Sbornik, 10, (1970), 217-273 (in English).
		
		%	\bibitem{LWR1}
		%	\newblock M. J. Lighthill and  G. B. Whitham,
		%	\newblock On kinematic waves. II. A theory of traffic flow on long crowded roads.
		%	\newblock {\it Proc. Roy. Soc. London Ser. A} 229 (1955), 317-345.
		
		
		\bibitem{lax1}
		\newblock  P. D. Lax,
		\newblock  Hyperbolic systems of conservation laws. II,
		\newblock{\it Comm. Pure Appl. Math.,} 10 (1957)  537-566.
		
		
		\bibitem{LPT} P.-L.~Lions, B.~Perthame, and E.~Tadmor.
		\newblock A kinetic formulation of  multidimensional scalar conservation laws and related equations.
		\newblock {\it J. Amer. Math. Soc.} 7, 169-192, (1994).
		
		
		
		\bibitem{love}
		\newblock E. R. Love and L. C. Young,
		\newblock Sur une classe de fonctionnelles lin\'{e}aires,
		\newblock {\it Fund. Math.,} 28 (1937), 243-257.
		
		%	\bibitem{EM}  E. Marconi,
		%\newblock Regularity estimates for scalar conservation laws in one space dimension,
		%\newblock {\it J. Hyperbolic Differ. Eq.}, 15 (4):623-691, (2018).
		%	
		\bibitem{traffic}
		\newblock S. Mochon,
		\newblock An analysis for the traffic on the highways with changing surface condition,
		\newblock {\it Math. Model.}, 9 (1987), no. 1, 1-11.
		
		\bibitem{MO}
		\newblock J. Musielak and W. Orlicz,
		\newblock On space of functions of finite generalized variation,
		\newblock {\it Bull. Acad. Pol. Sc.}, 5 (1957), 389-392.
		
		\bibitem{MO1}
		\newblock J. Musielak and  W. Orlicz,
		\newblock On generalized variations I,
		\newblock {\it studia mathematica XVIII,} (1959), 11-41.
		
		\bibitem{P05} E. Y. Panov.,
		\newblock Existence of strong traces for generalized solutions of multidimensional scalar conservation laws,
		\newblock {\it J. Hyperbolic Differ. Equ.}  2, no. 4, 885--908, (2005).
		
		\bibitem{P07} E. Y. Panov,
		\newblock Existence of strong traces for quasi-solutions of multidimensional conservation laws,
		\newblock {\it J. Hyperbolic Differ. Equ. } 4 (4), 729--770, (2007). 
		
		\bibitem{P09} E. Y. Panov,
		\newblock
		On existence and uniqueness of entropy solutions to the Cauchy problem for a conservation law with discontinuous flux. 
		\newblock 
		{\it J. Hyperbolic Differ. Equ. } 6, No. 3, 525-548 (2009)
		
		\bibitem{ol}
		\newblock  O. A. Ole\u{\i}nik,
		\newblock  Discontinuous solutions of non-linear differential equations, (Russian),
		\newblock {\it  Uspehi Mat. Nauk (N.S.),} 12 (1957) no. 3, (75), 3-73.
		
		%	\bibitem{LWR2}
		%	\newblock P.I. Richards,
		%	\newblock Shock waves on the highway.
		%	\newblock {\it Operations Res.,} 4, 42-51(1956).
		\bibitem{ion}
		\newblock D. S. Ross,
		\newblock Two new moving boundary problems for scalar conservation laws,
		\newblock {\it Comm. Pure Appl. Math.}, 41 (1988), no.5, 725-737.
		
		%	\bibitem{smoller}
		%	\newblock  J. Smoller,
		%	\newblock  Shock waves and reaction-diffusion equations, Second edition,
		%	\newblock {\it 258. Springer-Verlag, New York,} 1994.
		
		\bibitem{tower}
		\newblock J. D. Towers,
		\newblock Convergence of a difference scheme for conservation laws with a discontinuous flux,
		\newblock {\it SIAM J. Numer. Anal.}, 38 (2000), no. 2, 681-698.
		
		%	\bibitem{tower1}
		%	\newblock  J. D. Towers,
		%	\newblock A difference scheme for conservation laws with a discontinuous flux: the nonconvex case,
		%	\newblock {\it SIAM J. Numer. Anal.}, 39 (2001), no. 4, 1197-1218.
		%	
		
		\bibitem{Volpert}	
		\newblock   A. I., Vol'pert. 
		\newblock  Spaces $BV$ and quasilinear equations. 
		\newblock (Russian){\it  Mat. Sb. (N.S.)} 73 (115) 1967, 255–302. 
		
		
		
		
	\end{thebibliography}
\end{document}